\documentclass[a4paper,11pt,oneside]{article}
\usepackage[a4paper, margin=1in]{geometry}
\usepackage{setspace}
\usepackage{textcomp}
\usepackage{enumerate}
\usepackage{amsmath, amssymb,amsfonts,amsthm,mathrsfs}
\usepackage{algpseudocode,algorithm}
\usepackage{cancel, cases}
\usepackage{soul}
\usepackage{float,graphicx,subcaption,epstopdf}
\usepackage{longtable,tabu,threeparttable,rotating,adjustbox,multicol,multirow}
\usepackage[normalem]{ulem}
\usepackage[dvipsnames]{xcolor}
\usepackage{tcolorbox} 

\usepackage{hyperref}

\usepackage{setspace,caption}

\usepackage[round, authoryear]{natbib}

\newcommand\tr{\text{tr}}

\newcommand{\bi}{\begin{itemize}}
\newcommand{\ei}{\end{itemize}}

\theoremstyle{plain}
\newtheorem{theorem}{Theorem}
\newtheorem{corollary}[theorem]{Corollary}
\newtheorem{lemma}[theorem]{Lemma}
\newtheorem{proposition}[theorem]{Proposition}

\theoremstyle{definition}

\newtheorem{assumption}[theorem]{Assumption}

\theoremstyle{remark}
\newtheorem{remark}{Remark}

\numberwithin{equation}{section}
\numberwithin{theorem}{section}

\title{Exploratory Randomization for Discrete-Time Linear Exponential Quadratic Gaussian (LEQG) Problem}
\author{S\'ebastien Lleo and Wolfgang Runggaldier}
\date{\today}

\begin{document}
\maketitle

\begin{abstract}
We investigate exploratory randomization for an extended linear-exponential-quadratic-Gaussian (LEQG) control problem in discrete time. This extended control problem is related to the structure of risk-sensitive investment management applications. We introduce exploration through a randomization of the control. Next, we apply the duality between free energy and relative entropy to reduce the LEQG problem to an equivalent risk-neutral LQG control problem with an entropy regularization term, see, e.g. \citet{daipraConnectionsStochasticControl1996}, for which we present a solution approach based on Dynamic Programming. Our approach, based on the energy-entropy duality may also be considered as leading to a justification for the use, in the literature, of an entropy regularization when applying a randomized control.
\end{abstract}

\maketitle


\textbf{keywords}: risk-sensitive control, free energy-entropy duality, stochastic differential games, Linear Exponential Quadratic Gaussian control, Linear Quadratic Gaussian games, exploratory controls.

\section{Introduction}

Real-world stochastic control problems are often characterized by incomplete information about the state or the model.  Traditionally, filtering has been used to address incomplete information \citep[see][]{bensoussanStochasticControlPartially1984,bensoussanApproximationMethodStochastic1987}. More recently, reinforcement learning has emerged as a promising alternative \citep[see][]{wangReinforcementLearningContinuous2020,jiaQLearningContinuousTime,jiaPolicyEvaluationTemporalDifference2022,jiaPolicyGradientActor2022,hamblyPolicyGradientMethods2021}. These techniques rely on iterative interactions with the environment, balancing exploration—randomized actions to acquire information—and exploitation—actions that maximize a reward function based on available data. Although reinforcement learning evolved from Markov Decision Processes, integrating it with stochastic control remains challenging. \citet{wangReinforcementLearningContinuous2020} made a key advance in this direction by showing how relaxed controls create exploratory policies inside standard (risk-neutral) control problems formulated in continuous time. Adding an entropy penalization/regularization to the control criterion further led to a trade-off between exploration and exploitation.

A particularly important class of stochastic control problems is that of risk-sensitive control \citep[see][]{Whittle1990,befrna98,bieleckiRiskSensitiveMarkovDecision2022a,bauerleMoreRiskSensitiveMarkov2014a,bauerleMarkovDecisionProcesses2024}, with linear exponential of quadratic Gaussian (LEQG) problems as a notable subclass \citep[see][]{ja73,whittleRiskSensitiveLinearQuadratic1981a,bevs85,Bensoussan1992}. LEQG problems are especially relevant in financial applications \citep[see e.g.][and references therein]{bipl99,kuna02,dall_RSBench,DavisLleoBook2014,piteraDiscretetimeRiskSensitive2023,dall_BBL_2020,davisRisksensitiveBenchmarkedAsset2021,LleoRunggaldierSeparation24}.

In this paper, we consider a discrete-time risk-sensitive control problem of the LEQG type drawing inspiration from the results of \citet{wangReinforcementLearningContinuous2020}.  Specifically, we introduce randomized controls to emulate exploration.  We generate these randomized controls by perturbing a deterministic control with an additive Gaussian distribution. Here, the mean is the deterministic control that captures exploitation, while the variance, which provides an additional (antagonizing) control, models exploration. While implementing additive randomization in continuous-time settings poses significant technical challenges\footnote{It is worth noting that alternative continuous-time approaches, such as those developed by Huy\^en Pham and coauthors (for a recent contribution see e.g. \citet{denkertControlRandomisationApproach2024}), provide valuable insights.}, the discrete-time framework allows us to focus more directly on the exploration aspect.

Our main contribution is to solve the randomized LEQG problem using a duality between free energy and entropy \citet{daipraConnectionsStochasticControl1996}. This approach reduces the risk-sensitive randomized stochastic control problem to a risk-neutral stochastic game problem with an entropy penalization. This penalization comes from two sources: (i) the randomized controls and (ii) the transition from risk-sensitive to risk-neutral.  While \citet{wangReinforcementLearningContinuous2020} introduce an entropy penalization term as a regularisation, the entropy penalization in our approach arises naturally from the energy-entropy duality, providing a conceptual justification for its use. Unlike Wang et al., who remain in the domain of standard risk-neutral stochastic control problems, our formulation yields a stochastic game. In this game, the decision-maker selects controls to exploit available information, while an antagonistic player determines the variance of the control perturbations, governing the degree of exploration. 

The introduction of randomization into stochastic control problems is primarily motivated by its applicability in reinforcement learning. A possible methodology to implement reinforcement learning in our setup is via a policy gradient approach, whereby learning is applied to a parametrized optimal policy \citep[see for instance][]{hamblyPolicyGradientMethods2021}.  However, our paper focuses on the foundational modeling steps that precede the development of the reinforcement learning algorithm. Specifically, we demonstrate how to reduce the randomized LEQG problem to a randomized stochastic LQG game via the energy-entropy duality. We also derive the optimal policy and value function. Both turn out to be parametrized by the initial data of the model via a backward recursion, as stated in Theorem \ref{theo:main:recursions}. If the model parameters are unknown and are updated recursively by some method that we only sketch briefly, the backward recursion has to be recomputed. In this sense, our approach retains conceptual similarities with the policy gradient framework, and in particular, actor-critic methods as further illustrated in Section \ref{S.4.3} below.

 \section{Setup}\label{sec:setup}

Consider a discrete-time grid with times $t=0,1, \ldots, T$, where $T < \infty$, and let $\left(\Omega,  \mathcal{F}, \left(\mathcal{F}\right)_{t = 0, \ldots, T}, \mathbb{P} \right)$ be a filtered complete probability space, where $\mathcal{F}$ is the natural filtration and $\mathbb{P}$ is a reference probability measure. 

\subsection{Linear Exponential Quadratic (LEQG) Control Problem}

We start with a classical LEQG problem with standard noise under the measure $\mathbb{P}$. The dynamics of the controlled state $x_t \in \mathbb{R}^{d_x}, d_x \in \mathbb{N}$ is
\begin{align}\label{eq:state}
    x_{t+1} = a + A x_t + B u_t +w_t, \quad t=0, \ldots, T-1,
\end{align}
where  $a \in \mathbb{R}^{d_x}$, $A \in \mathbb{R}^{d_x\times d_x}$, $B \in \mathbb{R}^{d_x\times d_u}$. Here, $u_t \in \mathbb{R}^{d_u}, d_u \in \mathbb{N}$ is the control and $w_t$ is a $d_x$-dimensional Gaussian system noise, that is $w_t \sim \mathcal{N}\left(0,\Lambda_t \right)$ where $\Lambda_t  \in \mathbb{R}^{d_x \times d_x}, t=0,\ldots, T-1$. The system noises are serially independent so $\mathbf{E} \left[ (w_t)'w_s \right] = 0, t \neq s$. 

The risk-sensitive criterion $J$ to be maximized is defined as
\begin{align}\label{eq:criterion:J}
    J(u;T,\theta) = -\frac{1}{\theta} \ln \mathbf{E} \left[ e^{\theta G_T} \right],
\end{align}
where $\theta \in (-1,0) \cup (0,\infty)$ is the risk-sensitivity, $G_T$ is the cost defined as
\begin{align}\label{eq:reward}
    G_T := \sum_{t=0}^{T-1} \left( x_t' M x_t  + u_t' N_t u_t  + u_t' Q x_t + x_t' m  + u_t' n \right)  + x_T' M_T x_T + x_T' m_T,   
\end{align}
where $M \in \mathbb{R}^{d_x \times d_x}$ is semipositive definite and symmetric, $N_t \in \mathbb{R}^{d_u \times d_u}, t=1,\ldots,T-1$ is also semipositive definite and symmetric, $Q \in \mathbb{R}^{d_u \times d_x}$, $m \in \mathbb{R}^{d_x}$, $n \in \mathbb{R}^{d_u}$, $M_T \in \mathbb{R}^{d_x \times d_x}$, and $m_T \in \mathbb{R}^{d_x}$. For convenience, we also introduce the exponentially transformed criterion $I$ defined as 
\begin{align}\label{eq:criterion:I}
    I(u;T,\theta) = e^{-\theta J(u;T,\theta)} = \mathbf{E} \left[ e^{\theta G_T} \right],
\end{align}
which is to be minimized.

\begin{remark}
The standard formulation of risk-sensitive control problems does not contain a constant $a$ in the state dynamics or a cross-term $u_t Q x_t$ in the running cost. These additions are required in investment management applications of continuous-time risk-sensitive control, see for example \citet{DavisLleoBook2014}. Including them ensures that our model can be used to solve discretized versions of the risk-sensitive asset management model by \citet{bipl99,kuna02} and of the risk-sensitive benchmarked asset management model by \citet{dall_RSBench}.
\end{remark}

\subsection{Exploration via Randomized Controls}

The classical LEQG model may not provide a perfect description of reality, and we may not have sufficient data to estimate our model parameters with a high degree of accuracy. We employ exploratory randomized controls to address these issues. \citet{wangReinforcementLearningContinuous2020} already proposed using randomized controls in continuous-time stochastic control problems to create the kind of exploration framework that is already customary in reinforcement learning methods. These authors showed that under suitable assumptions, adding an entropy-related penalization to their stochastic control framework produces an optimal exploration-exploitation trade-off.  

Our model achieves randomized exploration by perturbing a deterministic control $\bar{u}_t$, that will be determined by optimization, with a random perturbation $v_t \in \mathbb{R}^{d_u}$. This perturbation is Gaussian, that is,  $v_t \sim \mathcal{N} \left(0, \Xi_t\right)$ where the covariance matrix $\Xi_t \in \mathbb{R}^{d_u \times d_u}, t=0,\ldots,T-1$, and $v_t$ is serially independent. Hence, the randomized control is of the form $u_t = \bar{u}_t + v_t \sim \mathcal{N} \left(\bar{u}_t, \Xi_t\right)$  and it admits a representation as a measure $\pi\left(du;\bar{u}_t\right) \sim \mathcal{N} \left(\bar{u}_t, \Xi_t\right)$.

\begin{remark}
    A motivation for letting the covariance matrix $\Xi$ depend on time is to have the possibility of reducing the breadth of exploration over time, consistently with the reinforcement learning literature. Exploration is more valuable in early times when little information is available. As more information becomes available, exploration gradually loses its value. A further motivation will be mentioned in the Conclusions section \ref{S.4.4} below.
    
\end{remark}

Taking into consideration the randomization of controls, we express the elements $\omega$ of the underlying sample space $\Omega$ as $$\omega = \left(\omega^w,\omega^v\right):= \left(w_0, \ldots,w_{T-1}, v_0, \ldots, v_{T-1}\right).$$

Next, we rewrite the state dynamics at \eqref{eq:state} as
\begin{align}\label{eq:state:rand}
    x_{t+1} 
    &= a + A x_t + B \left(\bar{u}_t + v_t \right)+w_t,
\end{align}
for $t=0, \ldots, T-1$. Inspired by Wang et al., we express the cost as: 
\begin{align}\label{eq:reward:rand}
    G_T 
    :=& \sum_{t=0}^{T-1} \left[ x_t' M x_t  + \int u_t' N_t u_t \pi\left(du;\bar{u}_t\right) + \int u_t' Q x_t \pi\left(du;\bar{u}_t\right) 
    \right.
                                \\
    & \left. 
    + x_t' m
    + \int u_t' n \pi\left(du;\bar{u}_t\right) \right]  
    + x_T' M_T x_T + x_T' m_T
                                \nonumber\\
    =& \sum_{t=0}^{T-1} \left[ x_t' M x_t  + \tr \left(\Xi_t N_t\right) + \bar{u}_t' N_t \bar{u}_t  + \bar{u}_t ' Q x_t + x_t' m  + \bar{u}_t' n \right]  
                                \nonumber\\
    & + x_T' M_T x_T + x_T' m_T.
                                \nonumber
\end{align}

\section{Solving the Randomized LEQG Control Problem via the Free Energy-Entropy Duality}\label{S.3}

We start this section by recalling key results from \citet{daipraConnectionsStochasticControl1996} about the duality between free energy and entropy. We shall next use the energy-entropy duality to associate to the given randomized LEQG problem a risk-neutral randomized stochastic game with a penalization given by an appropriate relative entropy.  This penalization acts on both the transition from a risk-sensitive to risk-neutral problem and the randomization of controls.


As a reference for the reader, Table~\ref{tab:variables:summary} lists the key variables and parameters used in Sections 2 and 3.

\begin{table}
    \centering
    \resizebox{\textwidth}{!}{\begin{tabular}{cclc}
        \hline\\
        Variable/ & Dimension   & Definition & Introduced \\
        Parameter & &  & in Section \\
        \hline\\

         \multicolumn{4}{l}{\underline{State and total cost:}} \\
         
         $x_t$  & $\mathbb{R}^{d_x}$ & State variable & 2.1\\
         
         $G_T$  & $\mathbb{R}$ & Total cost & 2.1 \\
        
         \hline\\
         
         \multicolumn{2}{l}{\underline{Under the measure $\mathbb{P}$:}} & 
         \multicolumn{2}{l}{$x_{t+1} = a + A x_t + B u_t +w_t$}\\

         \multicolumn{2}{l}{ } &
         \multicolumn{2}{l}{\small $G_T = \sum_{t=0}^{T-1} \left( x_t' M x_t  + u_t' N_t u_t  + u_t' Q x_t + x_t' m  + u_t' n \right)$} \\
         \multicolumn{2}{l}{ } &
         \multicolumn{2}{l}{\small $\phantom{G_T =} + x_T' M_T x_T + x_T' m_T$} \\

         \\
        
         $w_t$ & $\mathbb{R}^{d_x}$ & State noise: $w_t \sim \mathcal{N}\left(0,\Lambda_t \right)$ & 2.1 \\
   
         $u_t$ & $\mathbb{R}^{d_u}$ & 1. Generic control for the initial LEQG problem & 2.1 \\
                  &  & 2. Exploratory control: $u_t = \bar{u}_t + v_t \sim \mathcal{N} \left(\bar{u}_t, \Xi_t\right)$  & 2.2 \\
         $\bar{u}_t$ & $\mathbb{R}^{d_u}$ & Deterministic control & 2.2 \\
         $v_t$   & $\mathbb{R}^{d_u}$ & Exploration noise: & 2.2 \\
                &  & $v_t \sim \mathcal{N}\left(0,\Xi_t \right)$ & \\

                  \hline\\
         \multicolumn{2}{l}{\underline{Under the measure $\mathbb{P}^{\bar\gamma,\bar\eta}$:}} &
         \multicolumn{2}{l}{$x_{t+1} 
= a + A x_t + B \left(\bar{u}_t + v^{\bar\eta}_t \right)+ w^{\bar\gamma}_t$}\\

         \multicolumn{2}{l}{ } &
         \multicolumn{2}{l}{\small $G_T = \sum_{t=0}^{T-1} \Big[
        x_t' M x_t  
        + \tr \left( \Xi_t N_t \right) + \left( \bar{u}_t + \bar\eta_t \right)'N_t\left( \bar{u}_t + \bar\eta_t \right)
                \Big]$} \\
        \multicolumn{2}{l}{ } &
        \multicolumn{2}{l}{\small $\phantom{G_T =} + \left( \bar{u}_t + \bar\eta_t \right)' Q x_t 
        + x_t' m  
        + \left( \bar{u}_t + \bar\eta_t \right)' n 
        \Big]$} \\
        \multicolumn{2}{l}{ } &
        \multicolumn{2}{l}{\small $\phantom{G_T =} + x_T' M_T x_T + x_T' m_T$} \\

         \\ 
         
         $\bar{\gamma}$ & $\mathbb{R}^{d_x}$  & Change of measure process (control) & 3.2\\
         
         $\bar{\eta}$   & $\mathbb{R}^{d_u}$ & Change of measure process (control) & 3.2\\
         
         $w^{\bar\gamma}_t$ & $\mathbb{R}^{d_x}$ & State noise: $w^{\bar\gamma}_t\sim \mathcal{N} \left(\bar\gamma_t, \Lambda_t \right)$ & 3.2\\

         $v^{\bar\eta}_t$ & $\mathbb{R}^{d_x}$ & Exploration noise: $v^{\bar\eta}_t \sim \mathcal{N} \left(\bar\eta_t, \Xi_t\right)$ & 3.2\\

        \hline\\

        \multicolumn{4}{l}{\underline{Value function for the stochastic differential game:}} \\
        $V(T;\theta)$  & $\mathbb{R}$   & Value of the stochastic differential game: & 3.2 \\
        & & {\small $V(T;\theta) 
    := \inf_{\bar{u}}\sup_{\bar\gamma,\bar\eta}
    \mathbf{E}^{\bar\gamma, \bar\eta} \left[ 
        \theta G_T 
        -\frac{1}{2} \sum_{t=0}^{T-1} \left( \bar\gamma_t' \Lambda^ {-1}_t \bar\gamma_t + \bar\eta_t' \Xi_t^{-1} \bar\eta_t\right)
    \right]$} & \\
        $V_t(x_t)$ & $\mathbb{R}$ & Recursive quadratic expression: $V_t(x_t) = \frac{1}{2}x_t' P_t x_t + x_t' p_t + r_t$ & 3.3\\

         \hline\\
         \multicolumn{4}{l}{\underline{Parameters used in the resolution of the control problem:}} \\

         $\mathfrak{A}^{(1)}_{t+1}$ & $\mathbb{R}^{d_u} \times \mathbb{R}^{d_x}$ & $\mathfrak{A}^{(1)}_{t+1} = B'P_{t+1}A + \theta Q$ & 3.3 \\ 

         $\mathfrak{a}^{(1)}_{t+1}$ & $\mathbb{R}^{d_u}$  & $\mathfrak{a}^{(1)}_{t+1} = B'P_{t+1} a + B' p_{t+1} + \theta  n$ & 3.3 \\
         
         $\mathfrak{A}^{(2)}_{t+1}$ & $\mathbb{R}^{d_x} \times \mathbb{R}^{d_x}$ & $\mathfrak{A}^{(2)}_{t+1} = P_{t+1}'A$ & 3.3 \\

         $\mathfrak{a}^{(2)}_{t+1}$ & $\mathbb{R}^{d_x}$ & $\mathfrak{a}^{(2)}_{t+1} = P_{t+1}'a + p_{t+1}$ & 3.3\\

         $\mathfrak{B}^{(1)}_{t+1}$ & $\mathbb{R}^{d_u} \times \mathbb{R}^{d_u}$    & $\mathfrak{B}^{(1)}_{t+1} = B' P_{t+1} B + 2\theta N_t$   & 3.3 \\
        
         $\mathfrak{B}^{(2)}_{t+1}$ & $\mathbb{R}^{d_x} \times \mathbb{R}^{d_x}$ & $\mathfrak{B}^{(2)}_{t+1} = - \Lambda^ {-1}_t + P_{t+1}$ & 3.3 \\

         $\mathfrak{B}^{(3)}_{t+1}$ & $\mathbb{R}^{d_u} \times \mathbb{R}^{d_u}$ & $\mathfrak{B}^{(3)}_{t+1} = -\Xi_t^{-1} + \mathfrak{B}^{(1)}_{t+1}=-\Xi_t^{-1}+  B' P_{t+1} B + 2\theta N_t $ & 3.3\\
         
         $\mathfrak{C}_{t+1}$   & $\mathbb{R}^{d_u} \times \mathbb{R}^{d_x}$ & $\mathfrak{C}_{t+1} = B'P_{t+1}$ & 3.3 \\
         
         $\mathfrak{G}_{t+1}$ & $\mathbb{R}^{d_x} \times \mathbb{R}^{d_x}$ & $\mathfrak{G}_{t+1} := \mathfrak{C}_{t+1}'\left(\mathfrak{B}^{(1)}_{t+1}\right)^{-1}\mathfrak{C}_{t+1} - \mathfrak{B}^{(2)}_{t+1}$ & 3.3 \\
         \hline
    \end{tabular}}
    \caption{Key variables and parameters appearing in Sections 2 and 3}
    \label{tab:variables:summary}
\end{table}

\subsection{General Definitions}\label{sec:energy-entropy:defs}

The \emph{free energy} of a random variable $\psi$ with respect to a reference measure $\mathbb{P}$ is, under suitable integrability conditions, 
\begin{align}
  \mathcal{E}^{\mathbb{P}}\left\{\psi\right\} = \ln \left(\int e^\psi  d\mathbb{P}\right).
\end{align} 
Consider a further measure $\mathbb{P}^\gamma$. The \emph{relative entropy} of $\mathbb{P}^\gamma$ with respect to $\mathbb{P}$ is
\begin{align}
  D_{\text{KL}}\left(\mathbb{P}^\gamma \| \mathbb{P} \right)
  = \mathbf{E}^\gamma \left[ \ln \left( \frac{d\mathbb{P}^\gamma}{d\mathbb{P}} \right)\right]
  ,
\end{align}
where $\mathbf{E}^\gamma \left[ \cdot \right]$ denotes the expectation with respect to the measure $\mathbb{P}^\gamma$. The \emph{free energy-entropy duality} relation \citep[see (ii) in Proposition 2.3. in ][]{daipraConnectionsStochasticControl1996} is
\begin{align}
  \mathcal{E}^{\mathbb{P}}\left\{\psi\right\} 
  = \sup_{\mathbb{P}^\gamma} \left\{ 
    \int \psi d\mathbb{P}^\gamma -  D_{\text{KL}}\left(\mathbb{P}^\gamma \| \mathbb{P} \right)
  \right\}.
\end{align}
\begin{remark}
Dai Pra et al. also provide an analytical formula that the optimal Radon-Nikodym derivative $\frac{d\mathbb{P}^{\gamma^*}}{d\mathbb{P}}$  must satisfy:
\begin{align}\label{eq:RNderivative:gamma:formula}
        \frac{d\mathbb{P}^{\gamma^*}}{d\mathbb{P}}
 = \frac{e^\Psi}{\mathbf{E}^{\mathbb{P}} \left[ e^\Psi \right]},
\end{align}
under suitable conditions presented in \citet{daipraConnectionsStochasticControl1996}. 
\end{remark}

\subsection{Energy-Entropy Duality for the Randomized LEQG Problem}

The energy-entropy duality associates with a given risk-sensitive LEQG problem, set with respect to an initial measure $\mathbb{P}$, a risk-neutral randomized stochastic control problem, set with respect to a transformed measure. This risk-neutral randomized stochastic control problem is penalized by an appropriate relative entropy term and formulated as a stochastic game.

This section aims to establish the relation between the risk-sensitive problem and the dual risk-neutral penalized problem. For this purpose, we introduce on the measurable space $\left(\Omega,  \mathcal{F}, \left(\mathcal{F}\right)_{t = 0, \ldots, T}\right)$ a probability measure $\mathbb{P}^{\bar\gamma,\bar\eta}$, different from $\mathbb{P}$, and parametrized by two processes $\bar\gamma = (\bar\gamma_t)_{t=0,\cdots,T} \in \mathbb{R}^{d_x}, \bar\eta = (\bar\eta_t)_{t=0,\cdots,T} \in \mathbb{R}^{d_u}$.

Under measure $\mathbb{P}^{\bar\gamma,\bar\eta}$, the state noise $w^{\bar\gamma}_t$ is  i.i.d.  Gaussian with mean $\bar\gamma_t$, which may depend on time, and covariance $\Lambda_t $, namely $w^{\bar\gamma}_t\sim \mathcal{N} \left(\bar\gamma_t, \Lambda_t \right)$.  Furthermore, the process $\bar\eta$ affects the randomized control $u_t = \bar{u}_t + v^{\bar\eta}_t$ via its noise $v^{\bar\eta}_t \sim \mathcal{N} \left(\bar\eta_t, \Xi_t\right)$, that is, the control noises are i.i.d.  Gaussian with mean $\bar\eta_t$ and covariance $\Xi_t$. We represent the randomized control as the measure $\pi\left(du;\bar{u}_t, \bar\eta_t\right) \sim \mathcal{N} \left(\bar{u}_t+\bar\eta_t, \Xi_t\right)$.

Denote by $\lambda$ the Lebesgue measure on $\left(\Omega,  \mathcal{F}, \left(\mathcal{F}\right)_{t = 0, \ldots, T}\right)$. By a slight abuse of notation, we denote by $w_t$ and $v_t$ the arguments in their respective noise densities, then 
\begin{align}\label{eq:MeasureChange:LebeguesToPbar}
    \frac{d\mathbb{P}^{\bar\gamma, \bar\eta}}{d\lambda }
    =& \prod_{t=0}^{T-1}  
        \left( \frac{1}{\sqrt{(2\pi)^{d_x} |\Lambda_t |}} \exp \left\{ -\frac{1}{2} \left( w_t - \bar\gamma_t \right)' \Lambda^ {-1}_t \left( w_t - \bar\gamma_t \right) \right\} \right)
                                \\
      & \times  
        \left( \frac{1}{\sqrt{(2\pi)^{d_u} |\Xi_t|}} \exp \left\{ -\frac{1}{2} \left( v_t - \bar\eta_t \right)' \Xi_t^{-1} \left( v_t - \bar\eta_t \right) \right\} \right).
                                \nonumber
\end{align}
The expression for $\frac{d\mathbb{P}}{d\lambda}$ takes the same form, but with $\bar\gamma_t = 0$ and $\bar\eta_t = 0$. Therefore,
\begin{align}\label{eq:MeasureChange:PbarToPcheck}
    \frac{d\mathbb{P}^{\bar\gamma, \bar\eta}}{d\mathbb{P}}
    =& \prod_{t=0}^{T-1}  
        \frac{\left( \cancel{\frac{1}{\sqrt{(2\pi)^{d_x} |\Lambda_t |}}} \exp \left\{ -\frac{1}{2} \left( w_t - \bar\gamma_t\right)' \Lambda^ {-1}_t \left( w_t - \bar\gamma_t\right) \right\} \right)}{\left( \cancel{\frac{1}{\sqrt{(2\pi)^{d_x} |\Lambda_t |}}} \exp \left\{ -\frac{1}{2} w_t' \Lambda^ {-1}_t w_t \right\} \right)}
                               \\
      & \times  
        \frac{\left( \cancel{\frac{1}{\sqrt{(2\pi)^{d_u} |\Xi_t|}}} \exp \left\{ -\frac{1}{2} \left( v_t - \bar\eta_t \right)' \Xi_t^{-1} \left( v_t - \bar\eta_t \right) \right\} \right)}{\left( \cancel{\frac{1}{\sqrt{(2\pi)^{d_u} |\Xi_t|}}} \exp \left\{ -\frac{1}{2} v_t' \Xi_t^{-1} v_t \right\} \right)}
                                \nonumber\\
    =& \prod_{t=0}^{T-1}  
        \exp \left\{ \bar\gamma_t' \Lambda^ {-1}_t w_t  -\frac{1}{2} \bar\gamma_t' \Lambda^ {-1}_t \bar\gamma_t\right\}
        \exp \left\{ \bar\eta_t' \Xi_t^{-1} v_t -\frac{1}{2} \bar\eta_t' \Xi_t^{-1} \bar\eta_t\right\}
                                \nonumber
\end{align}
The \emph{relative entropy} of $\mathbb{P}^{\bar\gamma, \bar\eta}$ with respect to $\mathbb{P}$ is
\begin{align}
  & D_{\text{KL}}\left(\mathbb{P}^{\bar\gamma, \bar\eta} \| \mathbb{P} \right)
                                        \\  
    =& \mathbf{E}^{\bar\gamma, \bar\eta} \left[ \ln \left( \frac{d\mathbb{P}^{\bar\gamma, \bar\eta}}{d\mathbb{P}} \right) \right]
                                        \nonumber\\
    =& \sum_{t=0}^{T-1} \mathbf{E}^{\bar\gamma, \bar\eta} \left[ \left\{ \bar\gamma_t' \Lambda^ {-1}_t w_t  -\frac{1}{2} \bar\gamma_t' \Lambda^ {-1}_t \bar\gamma_t\right\} + \left\{ \bar\eta_t' \Xi_t^{-1} v_t -\frac{1}{2} \bar\eta_t' \Xi_t^{-1} \bar\eta_t\right\} \right]  
                                        \nonumber\\
    =& \sum_{t=0}^{T-1} \left\{ \bar\gamma_t' \Lambda^ {-1}_t \mathbf{E}^{\bar\gamma, \bar\eta} \left[w_t\right]  -\frac{1}{2} \bar\gamma_t' \Lambda^ {-1}_t \bar\gamma_t\right\} + \left\{ \bar\eta_t' \Xi_t^{-1} \mathbf{E}^{\bar\gamma, \bar\eta} \left[v_t\right] -\frac{1}{2} \bar\eta_t' \Xi_t^{-1} \bar\eta_t\right\}
                                        \nonumber\\
     =& \sum_{t=0}^{T-1} \left\{ \bar\gamma_t' \Lambda^ {-1}_t\bar\gamma_t  -\frac{1}{2} \bar\gamma_t' \Lambda^ {-1}_t \bar\gamma_t\right\} + \left\{ \bar\eta_t' \Xi_t^{-1}  \bar\eta_t -\frac{1}{2} \bar\eta_t' \Xi_t^{-1} \bar\eta_t\right\}
                                        \nonumber\\
    =& \frac{1}{2} \sum_{t=0}^{T-1} \left( \bar\gamma_t' \Lambda^ {-1}_t \bar\gamma_t + \bar\eta_t' \Xi_t^{-1} \bar\eta_t\right)
                                        \nonumber
\end{align}


Under $\mathbb{P}^{\bar\gamma, \bar\eta}$, the state dynamics at \eqref{eq:state} becomes
\begin{align}\label{eq:state:Pbar:rand}
    x_{t+1} 
&= a + A x_t + B \left(\bar{u}_t + v^{\bar\eta}_t \right)+ w^{\bar\gamma}_t,
\end{align}
for $t=0, \ldots, T-1$. Furthermore, the cost function at \eqref{eq:reward:rand} can be expressed as:
\begin{align}\label{eq:reward:rand:Pbar}
    G_T 
    :=& \sum_{t=0}^{T-1} \left[ 
        x_t' M x_t  
        + \int u_t' N_t u_t \pi\left(du;\bar{u}_t,\bar\eta_t\right) 
        + \int u_t' Q x_t \pi\left(du;\bar{u}_t,\bar\eta_t\right) 
        \right.
                                \\
        &\left.
        + x_t' m  
        + \int u_t' n \pi\left(du;\bar{u}_t,\bar\eta_t\right) \right]  
        + x_T' M_T x_T + x_T' m_T
                                \nonumber\\
    =& \sum_{t=0}^{T-1} \left[ 
        x_t' M x_t  
        + \tr \left( \Xi_t N_t \right) + \left( \bar{u}_t + \bar\eta_t \right)'N_t\left( \bar{u}_t + \bar\eta_t \right) 
        \right.
                                \nonumber\\
    & \left.
        + \left( \bar{u}_t + \bar\eta_t \right)' Q x_t
        + x_t' m  
        + \left( \bar{u}_t + \bar\eta_t \right)' n 
        \right]
        + x_T' M_T x_T + x_T' m_T.
                                \nonumber
\end{align}

 On the basis of the \emph{energy-entropy duality}, we then write 
\begin{align}\label{eq:criterion:I:Pbar:step1}
    \ln I(\bar{u};T,\theta) 
    = \sup_{\bar\gamma,\bar\eta}
    \mathbf{E}^{\bar\gamma, \bar\eta} \left[ 
        \theta G_T 
        -\frac{1}{2} \sum_{t=0}^{T-1} \left( \bar\gamma_t' \Lambda^ {-1}_t \bar\gamma_t + \bar\eta_t' \Xi_t^{-1} \bar\eta_t\right)
    \right],
\end{align}
where $I$ is the criterion defined at \eqref{eq:criterion:I}.

Taking the infimum over $\bar{u}$ to minimize the criterion $I$, we have
\begin{align}\label{eq:EEDuality:inf}
    & \inf_{\bar{u}}\ln I(\bar{u};T,\theta) 
    = \inf_{\bar{u}} \sup_{\bar\gamma,\bar\eta}
    \mathbf{E}^{\bar\gamma, \bar\eta} \left[ 
        \theta G_T 
        -\frac{1}{2} \sum_{t=0}^{T-1} \left( \bar\gamma_t' \Lambda^ {-1}_t \bar\gamma_t + \bar\eta_t' \Xi_t^{-1} \bar\eta_t\right)
    \right]
                            \nonumber\\
    \Leftrightarrow&
    \ln \inf_{\bar{u}} I(\bar{u};T,\theta) 
    = \inf_{\bar{u}} \sup_{\bar\gamma,\bar\eta}
    \mathbf{E}^{\bar\gamma, \bar\eta} \left[ 
        \theta G_T 
        -\frac{1}{2} \sum_{t=0}^{T-1} \left( \bar\gamma_t' \Lambda^ {-1}_t \bar\gamma_t + \bar\eta_t' \Xi_t^{-1} \bar\eta_t\right)
    \right]
                            \nonumber\\
    \Leftrightarrow&
    \inf_{\bar{u}} I(\bar{u};T,\theta) 
    = 
    \exp\left\{ \inf_{\bar{u}} \sup_{\bar\gamma,\bar\eta}
    \mathbf{E}^{\bar\gamma, \bar\eta} \left[ 
        \theta G_T 
        -\frac{1}{2} \sum_{t=0}^{T-1} \left( \bar\gamma_t' \Lambda^ {-1}_t \bar\gamma_t + \bar\eta_t' \Xi_t^{-1} \bar\eta_t\right)
    \right] \right\},
\end{align}
where the first equivalence follows from Lemma 5.3.1 in \citet{Meneghini1994}.

Focusing on the term inside the exponential on the right-hand side of \eqref{eq:EEDuality:inf}, we consider the game with optimal value  
\begin{align}\label{eq:criterion:I:Pbar}
    V(T;\theta) 
    := \inf_{\bar{u}}\sup_{\bar\gamma,\bar\eta}
    \mathbf{E}^{\bar\gamma, \bar\eta} \left[ 
        \theta G_T 
        -\frac{1}{2} \sum_{t=0}^{T-1} \left( \bar\gamma_t' \Lambda^ {-1}_t \bar\gamma_t + \bar\eta_t' \Xi_t^{-1} \bar\eta_t\right)
    \right]
\end{align}
so that
\begin{align}\label{criterion1}
 \inf_{\bar{u}}I(\bar{u};T,\theta)=\exp\left\{V(T;\theta)\right\}
 \end{align}

\begin{remark}
We can thus interpret the energy-entropy duality's $\inf\sup$ as a two-player game against Nature. The agent applies control $u$ to minimize the expectation while Nature (via the duality) applies control $\bar\nu := \begin{pmatrix} \bar{\gamma}' &\bar{\eta}'\end{pmatrix}'$ to maximize it.  Notice that we subtract the entropy from the criterion $I(u;T,\theta)$ to be minimized in $\bar u$; on the other hand, we maximize over the entropy parameters $\bar\gamma$ and $\bar\eta$.  In \citet{wangReinforcementLearningContinuous2020}, the authors simply add the randomized control entropy to the criterion to be minimized with respect to the originally given control as a reasonable way to proceed but without further explanation. Their entropy is not parametrized and therefore their problem does not become a stochastic differential game. By contrast, our setup results from the energy-entropy duality, justifying the presence of an additive entropy penalization.
\end{remark}

\subsection{Solving the Risk-Neutral Penalized Game Problem}

To solve the stochastic game with optimal value $ V(T;\theta)$ as in \eqref{eq:criterion:I:Pbar} we shall apply the Dynamic Programming Principle (DPP). To this effect, let $V_t(x_t)$ be the optimal value of the game at the generic time $t\quad(t=0,\cdots,T)$ when the controlled state takes the value $x_t$. The risk sensitivity parameter $\theta$ is supposed to be fixed. Using \eqref{eq:criterion:I:Pbar} that expresses the global optimal value of the game as well as \eqref{eq:reward:rand:Pbar} that expresses the cost $G_T$, the DPP lets us express $V_t$ recursively as
\begin{align}\label{eq:DPP:TC}
 V_T(x_T)=& \theta \,\left[x_T'M_Tx_T+x_T'm_T\right]
\end{align}
and, for $t=T-1,\cdots,1,0$, 
\begin{align}\label{eq:DPP:DPP} 
V_t(x_t)
    =&\inf_{\bar{u}}\sup_{\bar\gamma,\bar\eta} \mathbf{E}_{t,x_t}^{\bar\gamma, \bar\eta} \left\{
    \theta\left[ 
        x_t' M x_t  
        + \tr \left( \Xi_t N_t \right) + \left( \bar{u} + \bar\eta \right)'N_t\left( \bar{u} + \bar\eta \right)
    \right.\right.
                                     \\    
   & \left.\left. 
        + \left( \bar{u} + \bar\eta \right)' Q x_t 
        + x_t' m 
        + \left( \bar{u} + \bar\eta \right)' n \right]
    \right.
                                     \nonumber\\    
   & \left.
        -\frac{1}{2} \left( \bar\gamma' \Lambda^ {-1}_t \bar\gamma + \bar\eta' \Xi_t^{-1} \bar\eta\right)
        + V_{t+1}(x_{t+1}) \right\}
                                     \nonumber\\ 
    =& \inf_{\bar{u}}\sup_{\bar\gamma,\bar\eta} \left\{
    \theta \left[ 
        x_t' M x_t  
        + \tr \left( \Xi_t N_t \right) + \left( \bar{u} + \bar\eta \right)'N_t\left( \bar{u} + \bar\eta \right)
    \right.\right.
                                     \nonumber\\    
   & \left.\left. 
        + \left( \bar{u} + \bar\eta \right)' Q x_t 
        + x_t' m  
        + \left( \bar{u} + \bar\eta \right)' n \right]
    \right.
                                     \nonumber\\    
    &\left.
        -\frac{1}{2} \left( \bar\gamma' \Lambda^ {-1}_t \bar\gamma + \bar\eta' \Xi_t^{-1} \bar\eta\right)
        + \mathbf{E}_{t,x_t}^{\bar\gamma, \bar\eta} \left[V_{t+1}(x_{t+1}) \right]\right\}, 
                                    \nonumber
\end{align}
where $\mathbf{E}_{t,x_t}^{\bar\gamma, \bar\eta}$ denotes the expectation with respect to the measure $\mathbb{P}^{\bar\gamma, \bar\eta}$ given a generic time $t$ and a state process value $x_{t}$ at time $t$, and where we have used the fact that  $\Lambda_t $ and $\Xi_t$ are covariance matrices and therefore, $\Lambda_t  >0, \Xi_t >0$.

In what follows, we shall show that $V_t(x_t)$ has a quadratic expression in $x_t$  of the form of 
\begin{align}\label{eq:V:quadform0}
    V_t(x_t) = \frac{1}{2}x_t' P_t x_t + x_t' p_t + r_t,
\end{align}
and at the same time, we shall derive the expressions of a stationary point $(u_t^*,\gamma_t^*,\eta_t^*)$ as a candidate for the optimal control.  In the following subsection \ref{sec:saddleconditions},  we shall then present sufficient conditions for the stationary point to be a saddle point and thus to lead to optimal controls. 

Our main result in Theorem \ref{theo:main:recursions} below will be preceded by two propositions that, in turn, will be preceded by a lemma of independent interest. Theorem \ref{theo:main:recursions} will then follow as their corollary.

First we have
\begin{lemma}\label{L1}

Assuming that, at the generic time $t$, the optimal value $V_t(x_t)$ has a quadratic expression as in \eqref{eq:V:quadform0}, for a control triple $(\bar u,\bar\gamma,\bar\eta)$ we have 
\begin{align}\label{eq:V:quadform:xt}
    &\mathbf{E}_{t,x_t}^{\bar\gamma, \bar\eta} \left[V_{t+1}(x_{t+1})\right]
    \\
    =& 
        \frac{1}{2} \left(a + A x_t + B \bar{u}_t\right)' P_{t+1} 
        \left(a + A x_t + B \bar{u}_t \right)
        + \frac{1}{2} \bar\gamma_t' P_{t+1} \bar\gamma_t
                                        \nonumber\\
    &   + \frac{1}{2} \bar\eta_t'B'P_{t+1} B \bar\eta_t
        + \bar\eta_t' B' P_{t+1} \bar\gamma_t
        +\left(a + A x_t + B\bar{u}_t\right)' p_{t+1}
        + \bar\gamma_t' p_{t+1} 
                                        \nonumber\\
    &   + \bar\eta_t'B'p_{t+1}
        + \left(a + A x_t + B \bar{u}_t\right)'P_{t+1} \bar\gamma_t
        + \left(a + A x_t + B \bar{u}_t \right)' 
        P_{t+1} B \bar\eta_t
                                        \nonumber\\
    &+ \frac{1}{2} \tr\left( B'\Xi_t B P_{t+1} \right)
        + \frac{1}{2} \tr\left( \Lambda_t  P_{t+1} \right)
        + r_{t+1}.
                                        \nonumber
\end{align}

\end{lemma}

\begin{proof} We use the state value dynamics at \eqref{eq:state:Pbar:rand} to obtain an analytic expression for $\mathbf{E}_{t,x_t}^{\bar\gamma, \bar\eta} \left[V_{t_+1}(x_{t+1})\right]$ in terms of $x_t$. We have
\begin{align}
    \mathbf{E}_{t,x_t}^{\bar\gamma, \bar\eta} \left[V_{t+1}(x_{t+1})\right]
        =&  \mathbf{E}_{t,x_t}^{\bar\gamma, \bar\eta} \left[ 
         \frac{1}{2}\left(a + A x_t + B \left(\bar{u}_t + v^{\bar\eta}_t \right)+ w^{\bar\gamma}_t\right)' 
        P_{t+1} 
        \left(a + A x_t + B \left(\bar{u}_t + v^{\bar\eta}_t \right)+ w^{\bar\gamma}_t \right)
    \right. 
                                    \nonumber\\
    &\left.        
        + \left(a + A x_t + B \left(\bar{u}_t + v^{\bar\eta}_t \right)+ w^{\bar\gamma}_t\right)' p_{t+1} + r_{t+1}
    \right]
                                        \nonumber\\
    =&  \frac{1}{2} \mathbf{E}_{t,x_t}^{\bar\gamma, \bar\eta} \left[ 
        \left(a + A x_t + B \bar{u}_t\right)' 
        P_{t+1} 
        \left(a + A x_t + B \bar{u}_t \right) \right]
        + \frac{1}{2} \underbrace{\mathbf{E}_{t,x_t}^{\bar\gamma, \bar\eta} \left[ 
            (v^{\bar\eta}_t)'B' P_{t+1} B   v^{\bar\eta}_t
        \right]}_{
            = \tr\left( B'\Xi_tB P_{t+1} \right) 
            + \bar\eta_t'B'P_{t+1} B \bar\eta_t
            }     
                                    \nonumber\\
   &    + \frac{1}{2} \underbrace{\mathbf{E}_{t,x_t}^{\bar\gamma, \bar\eta} \left[ 
        (w^{\bar\gamma}_t)' P_{t+1} w^{\bar\gamma}_t
        \right]}_{
            = \tr\left( \Lambda_t  P_{t+1} \right)
            + \bar\gamma_t' P_{t+1} \bar\gamma_t
            }    
        + \mathbf{E}_{t,x_t}^{\bar\gamma, \bar\eta} \left[ 
        \left(a + A x_t + B \bar{u}_t \right)' 
        P_{t+1} B v^{\bar\eta}_t
        \right]
                                    \nonumber\\
    &    
        + \mathbf{E}_{t,x_t}^{\bar\gamma, \bar\eta} \left[ 
        \left(a + A x_t + B \bar{u}_t\right)' 
        P_{t+1} w^{\bar\gamma}_t 
        \right]    
        + \underbrace{\mathbf{E}_{t,x_t}^{\bar\gamma, \bar\eta} \left[ 
        (v^{\bar\eta}_t)'B' 
        P_{t+1} w^{\bar\gamma}_t
        \right]}_{= \bar\eta_t' B' P_{t+1} \bar\gamma_t}
                                    \nonumber\\
    &        
        + \left(a + A x_t + B\bar{u}_t\right)' p_{t+1} 
        + \mathbf{E}_{t,x_t}^{\bar\gamma, \bar\eta} \left[ (v^{\bar\eta}_t)'\right]B'p_{t+1}
        + \mathbf{E}_{t,x_t}^{\bar\gamma, \bar\eta} \left[ (w^{\bar\gamma}_t)'\right]p_{t+1}
        + r_{t+1}
                                    \nonumber
\end{align}
\begin{align}\label{eq:V:quadform:xt:proof}
    \mathbf{E}_{t,x_t}^{\bar\gamma, \bar\eta} \left[V_{t+1}(x_{t+1})\right]
    =&  \frac{1}{2} \left(a + A x_t + B \bar{u}_t\right)' 
        P_{t+1} 
        \left(a + A x_t + B \bar{u}_t \right)    
                                        \\
    &    \frac{1}{2} \tr\left( B'\Xi_t B P_{t+1} \right)
        + \frac{1}{2} \bar\eta_t'B'P_{t+1} B \bar\eta_t  
        + \frac{1}{2} \tr\left( \Lambda_t  P_{t+1} \right)     
                                    \nonumber\\
    &          
        + \frac{1}{2} \bar\gamma_t' P_{t+1} \bar\gamma_t  
        + \left(a + A x_t + B \bar{u}_t \right)' 
        P_{t+1} B \bar\eta_t      
                                    \nonumber\\
    &   + \left(a + A x_t + B \bar{u}_t\right)' 
        P_{t+1} \bar\gamma_t 
        + \bar\eta_t' B' P_{t+1} \bar\gamma_t  
                                    \nonumber\\
    &   + \left(a + A x_t + B\bar{u}_t\right)' p_{t+1}
        + \bar\eta_t'B'p_{t+1}
        + \bar\gamma_t' p_{t+1}
        + r_{t+1}
                                    \nonumber
\end{align}
from which the statement of the lemma follows.
\end{proof}

To ease notation, in what follows we shall use the shorthand notation, applied to the time point $t+1$ for a generic $t \in \left\{0,\ldots,T-1\right\}$.
\begin{align}\label{short}
    \mathfrak{A}^{(1)}_{t+1} &= B'P_{t+1}A + \theta Q,
    &
     \mathfrak{A}^{(2)}_{t+1} &= P_{t+1}'A \\
     \mathfrak{a}^{(1)}_{t+1}
    &= B'P_{t+1} a + B' p_{t+1} + \theta  n,
    &
    \mathfrak{a}^{(2)}_{t+1}
    &= P_{t+1}'a + p_{t+1}\nonumber\\
     \mathfrak{B}^{(1)}_{t+1} &= B' P_{t+1} B + 2\theta N_t &  \mathfrak{B}^{(2)}_{t+1}    
    &= - \Lambda^ {-1}_t + P_{t+1} 
                         \nonumber
   \end{align}

\begin{align}\label{short2}
    \mathfrak{B}^{(3)}_{t+1} 
    &= -\Xi_t^{-1} + \mathfrak{B}^{(1)}_{t+1}=-\Xi_t^{-1}+  B' P_{t+1} B + 2\theta N_t \\
 \mathfrak{C}_{t+1} &= B'P_{t+1} \nonumber\\
 \mathfrak{G}_{t+1}
    &= \mathfrak{C}_{t+1}'\left(\mathfrak{B}^{(1)}_{t+1}\right)^{-1}\mathfrak{C}_{t+1} - \mathfrak{B}^{(2)}_{t+1}
                                \nonumber
\end{align}

\begin{remark} Note that if $N_t$ is a symmetric matrix, as is the case in investment management problems, $\mathfrak{B}^{(1)}_{t+1} = B' P_{t+1} B + 2\theta N_t$ should also be symmetric because, as will be shown below,  $P_{t+1}$ 
is symmetric as the solution to a Riccati recursion. Similarly, $\mathfrak{B}^{(2)}_{t+1}$ is also symmetric. \end{remark} 

For a generic control triple $(\bar u,\bar\gamma,\bar\eta)$ and for given $t$, $x$, and a set of model parameters, in what follows, we shall also consider the function
\begin{align}\label{eq:auxfunc:F:shorthand}
    F(\bar{u}, \bar\gamma, \bar\eta)
    :=& 
        \frac{1}{2} \bar{u}' \mathfrak{B}^{(1)}_{t+1} \bar{u}
        + \bar{u}' \left(
            \mathfrak{A}^{(1)}_{t+1} x_t  
            + \mathfrak{a}^{(1)}_{t+1} 
          \right)
        + \frac{1}{2} \bar\gamma'\mathfrak{B}^{(2)}_{t+1} \bar\gamma
                                        \\
    &
        + \bar\gamma' \left( 
            \mathfrak{A}^{(2)}_{t+1} x_t
            + \mathfrak{a}^{(2)}_{t+1}
          \right)
        + \frac{1}{2} \bar\eta' \mathfrak{B}^{(3)}_{t+1}\bar\eta
        + \bar\eta' \left(
            \mathfrak{A}^{(1)}_{t+1} x_t 
            + \mathfrak{a}^{(1)}_{t+1} 
          \right)
                            \nonumber\\
    &   + \bar{u}'\mathfrak{C}_{t+1} \bar\gamma
        + \bar{u}' \mathfrak{B}^{(1)}_{t+1} \bar\eta
        + \bar\eta' \mathfrak{C}_{t+1} \bar\gamma,
                            \nonumber
\end{align}

To establish the main results, we consider the following conditions on the coefficients:

\begin{assumption}\label{as:coeffs}
For $t=0,\cdots,T-1$ assume
\begin{enumerate}[(i)]
    \item $\mathfrak{B}^{(1)}_{t+1} = B' P_{t+1} B + 2\theta N_t > 0 $;
    \item \begin{align*}- \begin{pmatrix}
		\mathfrak{B}^{(2)}_{t+1}  
		& \mathfrak{C}_{t+1}' \\ 
		\mathfrak{C}_{t+1}  
		& \mathfrak{B}^{(3)}_{t+1}    
        \end{pmatrix} 
        =- \begin{pmatrix}
		- \Lambda^ {-1}_t + P_{t+1}
		& (B'P_{t+1})' \\ 
		B'P_{t+1} 
		& -\Xi_t^{-1}+  B' P_{t+1} B + 2\theta N_t   
        \end{pmatrix} > 0.\end{align*}   
\end{enumerate}
    
\end{assumption}

We discuss these conditions in Section \ref{sec:saddleconditions}.

Using Lemma \ref{L1}, we prove two propositions on our way to establishing the main theorem.

\begin{proposition}\label{P1}
The optimal value function $V_t(x_t)$ at \eqref{eq:DPP:DPP} has, for $t=0,\cdots, T-1$, the following saddle point representation
 \begin{align}\label{eq:DPP:Vquadform1}
    V_t(x_t)
    =& \inf_{\bar{u}}\sup_{\bar\gamma,\bar\eta} \left\{ F(\bar{u}, \bar\gamma, \bar\eta)
    \right\}
        +   x_t' \left(\theta M + \frac{1}{2} A 'P_{t+1}A \right)x_t
                                        \\
    &
        + x_t'\left(
            A' P_{t+1} a + \theta m + A' p_{t+1}
          \right)
        + \frac{1}{2} \tr\left( B'\Xi_t B P_{t+1} \right)
        + \frac{1}{2} \tr\left( \Lambda_t  P_{t+1} \right)
        + r_{t+1} 
                                    \nonumber\\
    &        
        + \frac{1}{2} a'P_{t+1}a
        + a' p_{t+1}
        + \theta \tr \left( \Xi_t N_t \right).
                                        \nonumber
\end{align}

Furthermore, the candidate optimal controls are given by the stationary point  $(u_t^*,\gamma_t^*,\eta_t^*)$ of the quadratic function
$F(\bar{u}, \bar\gamma, \bar\eta)$ at \eqref{eq:auxfunc:F:shorthand} and can be explicitly expressed as

\begin{align}\label{eq:star:2} 
u_t^* &= \left(\mathfrak{B}^{(1)'}_{t+1}\mathfrak{B}^{(2)}_{t+1}- \mathfrak{C}'_{t+1} \mathfrak{C}_{t+1}\right)^{-1}
 \left[ \left(\mathfrak{A}^{(2)'}_{t+1}\mathfrak{C}_{t+1} -\mathfrak{A}^{(1)'}_{t+1}\mathfrak{B}^{(2)}_{t+1}\right)\,x_t  + \left(\mathfrak{C}'_{t+1}
    \mathfrak{a}^{(2)}_{t+1}- \mathfrak{B}^{(2)'}_{t+1} \mathfrak{a}^{(1)}_{t+1} \right)\right]
                                    \\
&=\left(B'P_{t+1}B+2\theta\,N-B'P'_{t+1}P_{t+1}B\right)^{-1}
 \left\{\left[P'_{t+1}A'BP_{t+1}-\left(B'P_{t+1} A+\theta\,Q\right)'\left(P_{t+1}-\Lambda^{-1}\right)\right]\,x_t\right.\nonumber\\
 &\left. 
 + P'_{t+1}BP_{t+1}a+P'_{t+1}Bp_{t+1}- \left(P_{t+1}-\Lambda^{-1}\right)\left(B'P_{t+1}a+B'p_{t+1}+\theta\,n\right)\right\}
 \nonumber\\ {}\nonumber\\
  \gamma_t^*&=\left(\mathfrak{B}^{(1)'}_{t+1}\mathfrak{B}^{(2)}_{t+1}- \mathfrak{C}'_{t+1} \mathfrak{C}_{t+1}\right)^{-1} 
 \left[ \left(\mathfrak{A}^{(1)'}_{t+1}\mathfrak{C}'_{t+1} -\mathfrak{A}^{(2)'}_{t+1}\mathfrak{B}^{(1)}_{t+1}\right)\,x_t  + \left(\mathfrak{C}_{t+1}' \mathfrak{a}^{(1)}_{t+1}- \mathfrak{B}^{(1)'}_{t+1} \mathfrak{a}^{(2)}_{t+1} \right)\right]\nonumber\\
 &\quad=\Bigl\{\left[\left(P_{t+1}-\Lambda^{-1}\right)\left(B'P_{t+1}B+2\,\theta\,N\right)-B'P_{t+1}P_{t+1}B\right]^{-1} \nonumber\\ 
 &\hspace{2cm}\left[B'P_{t+1}\left(A'P_{t+1}B+\theta\,Q\right)-A'P_{t+1}\left(B'P_{t+1}B+2\,\theta\,N\right)\right]\,x_t
 \nonumber \\
 &\hspace{2cm} + B'P_{t+1}\left(aP_{t+1}B+B'\,p_{t+1}+\theta\,n\right)-\left(B'P_{t+1}B+2\,\theta\,N\right)\left(aP_{t+1}+p_{t+1}\right)\Bigr\}
 \nonumber \\{}\nonumber\\
\eta_t^*&= -\left(\mathfrak{B}^{(3)}_{t+1}\right)^{-1}\,\left[ \mathfrak{A}^{(1)'}_{t+1} x_t + \mathfrak{a}^{(1)}_{t+1} + (\mathfrak{B}^{(1)'}_{t+1})u_t^*
  + \mathfrak{C}_{t+1} \gamma_t^*\right] \nonumber\\
  &  \Leftrightarrow \>\>-\Xi_t^{-1}\eta_t^*+ \mathfrak{A}^{(1)}_{t+1} x_t+ \mathfrak{a}^{(1)}_{t+1}-\left\{\mathfrak{A}^{(1)}_{t+1} x_t+ \mathfrak{a}^{(1)}_{t+1}
  + \mathfrak{C}_{t+1} \gamma_t^*\right\}+ \mathfrak{C}_{t+1} \gamma_t^* =0\nonumber\\
 &  \Leftrightarrow \>\>\eta_t^*=0
                                \nonumber
\end{align}
for $t=0,\ldots,T-1$.

\end{proposition}   
   
\begin{proof}
       
  Substituting \eqref{eq:V:quadform:xt} into the dynamic programming equation at \eqref{eq:DPP:DPP}, we obtain    
\begin{align}\label{eq:DPP:Vquadform2}
    & V_t(x_t)
                                        \nonumber\\ 
    =& \inf_{\bar{u}}\sup_{\bar\gamma,\bar\eta} \left\{
          \theta x_t' M x_t  
        + \theta \tr \left( \Xi_t N_t \right) 
        + \theta \bar{u}'N_t\bar{u}
        + 2\theta \bar{u}'N_t\bar\eta
        + \theta \bar\eta' N_t \bar\eta
        + \theta \bar{u}' Q x_t 
        + \theta \bar\eta' Q x_t
        + \theta \bar{u}' n
    \right.
                                        \nonumber\\    
    &\left.
        + \theta \bar\eta' n
        + \theta x_t' m
        -\frac{1}{2} \bar\gamma' \Lambda^ {-1}_t\bar\gamma
        -\frac{1}{2} \bar\eta' \Xi_t^{-1} \bar\eta
        + \frac{1}{2} \left(a + A x_t + B \bar{u}\right)' P_{t+1} 
        \left(a + A x_t + B \bar{u} \right)
    \right.
                                        \nonumber\\    
    &\left.
        + \frac{1}{2} \bar\gamma' P_{t+1} \bar\gamma
        + \frac{1}{2} \bar\eta'B'P_{t+1} B \bar\eta  
        + \left(a + A x_t + B\bar{u}\right)' p_{t+1}
        + \bar\gamma' p_{t+1} 
        + \bar\eta'B'p_{t+1}
        \right.
                                        \nonumber\\
    & \left. 
        + \left(a + A x_t + B \bar{u}\right)'P_{t+1} \bar\gamma
        + \left(a + A x_t + B \bar{u} \right)' 
        P_{t+1} B \bar\eta
        + \bar\eta_t' B' P_{t+1} \bar\gamma_t
        \right.
                                        \nonumber\\
    & \left.   
        + \frac{1}{2} \tr\left( B'\Xi_t B P_{t+1} \right)
        + \frac{1}{2} \tr\left( \Lambda_t  P_{t+1} \right)
        + r_{t+1} 
    \right\}
                                        \nonumber\\
    =& \inf_{\bar{u}}\sup_{\bar\gamma,\bar\eta} \left\{
          \theta x_t' M x_t  
        + \theta \tr \left( \Xi_t N_t \right) 
        + \theta \bar{u}'N_t\bar{u}
        + 2\theta \bar{u}'N_t\bar\eta
        + \theta \bar{u}' Q x_t 
        + \theta \bar\eta' Q x_t
        + \theta \bar{u}' n
        + \theta \bar\eta' n
    \right.
                                        \nonumber\\    
    &\left.
        + \theta x_t' m
        -\frac{1}{2} \bar\gamma' \left( \Lambda^ {-1}_t - P_{t+1} \right)\bar\gamma
        -\frac{1}{2} \bar\eta' \left( \Xi_t^{-1} - B'P_{t+1} B - 2 \theta N_t\right)\bar\eta
    \right.
                                        \nonumber\\
    & \left. 
        + \frac{1}{2} \bar{u}' B' P_{t+1} B \bar{u}
        +  \bar{u}' B' P_{t+1} \left(a + A x_t\right)
        + \frac{1}{2} a'P_{t+1}a
        + x_t'A' P_{t+1} a
        + \frac{1}{2} x_t'A 'P_{t+1}A x_t
    \right.
                                        \nonumber\\
    & \left.     
        + \left(a + A x_t \right)' p_{t+1}
        + \bar{u}'B' p_{t+1}
        + \bar\gamma' p_{t+1} 
        + \bar\eta'B'p_{t+1}
        + \left(a + A x_t\right)'P_{t+1} \bar\gamma_t
        + \bar{u}'B'P_{t+1} \bar\gamma_t
        \right.
                                        \nonumber\\
    & \left.   
        + \left(a + A x_t\right)' P_{t+1} B \bar\eta
        + \bar{u}' B' P_{t+1} B \bar\eta
        + \bar\eta_t' B' P_{t+1} \bar\gamma_t
        + \frac{1}{2} \tr\left( B'\Xi_t B P_{t+1} \right)
        + \frac{1}{2} \tr\left( \Lambda_t  P_{t+1} \right)
        + r_{t+1} 
    \right\} 
                                        \nonumber\\
    =& \inf_{\bar{u}}\sup_{\bar\gamma,\bar\eta} \left\{
            x_t' \left(\theta M + \frac{1}{2} A 'P_{t+1}A \right)x_t
        + x_t'\left(
            A' P_{t+1} a + \theta m + A' p_{t+1}
          \right)
    \right.
                                        \nonumber\\
    &\left.
        + \bar{u}' \left( \theta N_t + \frac{1}{2} B' P_{t+1} B \right)\bar{u}
        + \bar{u}' \left[
            \left( B'P_{t+1}A + \theta Q \right)x_t  
            + B'P_{t+1} a
            + B' p_{t+1}
            + \theta  n
          \right]
    \right.
                                        \nonumber\\
    &\left.
        -\frac{1}{2} \bar\gamma' \left( \Lambda^ {-1}_t - P_{t+1} \right)\bar\gamma
        + \bar\gamma' \left[ 
            P_{t+1}'A x_t
            + P_{t+1}'a
            + p_{t+1}
          \right]
    \right.
                                        \nonumber\\
    &\left.
        -\frac{1}{2} \bar\eta' \left( 
            \Xi_t^{-1} - B'P_{t+1} B - 2\theta N_t
          \right)\bar\eta
        + \bar\eta' \left[
            \left(B'P_{t+1}'A + \theta Q\right)x_t 
            + B'p_{t+1} 
            + B'P_{t+1}'a
            + \theta  n 
          \right]
    \right.
                                        \nonumber\\
    &\left.
        + \bar{u}'B'P_{t+1} \bar\gamma_t
        + \bar{u}'\left[ 
            B' P_{t+1} B 
            + 2\theta N_t
          \right] \bar\eta
        + \bar\eta_t' B' P_{t+1} \bar\gamma_t
    \right.
                                        \nonumber\\
    &\left.
        + \frac{1}{2} \tr\left( B'\Xi_t B P_{t+1} \right)
        + \frac{1}{2} \tr\left( \Lambda_t  P_{t+1} \right)
        + r_{t+1} 
        + \frac{1}{2} a'P_{t+1}a
        + a' p_{t+1}
        + \theta \tr \left( \Xi_t N_t \right)
    \right\}  
    \end{align}       
from which, by the definition of the function $F(\bar u,\bar\gamma,\bar\eta)$ in \eqref{eq:auxfunc:F:shorthand} the first part of the proposition follows.

On the basis of  \eqref{eq:DPP:Vquadform2},  we also define the Hamiltonian $\mathcal{H}$
\begin{align}\label{eq:Hamiltonian}
    & \mathcal{H}(t,x_t,\bar{u}_t,\bar\gamma_t,\bar\eta_t) 
                        \\
    :=& \theta \left[ 
        x_t' M x_t  
        + \text{tr} \left( \Xi N_t \right) + \left( \bar{u} + \bar\eta \right)'N_t\left( \bar{u} + \bar\eta \right)
        + \left( \bar{u} + \bar\eta \right)' Q x_t 
        \right.
                        \nonumber\\
    & \left.
            + x_t' m
            + \left( \bar{u} + \bar\eta \right)' n \right]
        -\frac{1}{2} \left( \bar\gamma' \Lambda^ {-1}_t \bar\gamma + \bar\eta' \Xi_t^{-1} \bar\eta\right)
    + \mathbf{E}_{t,x_t}^{\bar\gamma, \bar\eta} \left[V_{t+1}(x_{t+1}) \right]
                        \nonumber
\end{align}
where $\mathbf{E}_{t,x_t}^{\bar\gamma, \bar\eta}$ denotes the expectation with respect to the measure $\mathbb{P}^{\bar\gamma, \bar\eta}$ given a generic time $t$ and a state process value $x_{t}$ at time $t$.

The function $F$ is quadratic in the controls $\bar{u}$, $\bar\gamma$, and $\bar\eta$. Hence, it admits a unique stationary point $(u^*, \gamma^*, \eta^*)$, which is affine in the state $x_t$. Moreover, the Hessian is independent of the value of the controls.  Consequently, we can apply the first-order condition to $F$ in any ordering of the controls without affecting the outcome. Hence, the minimax condition
\begin{align}\label{eq:DPP:Vquadform3}
    \inf_{\bar{u}}\sup_{\bar\gamma,\bar\eta}
    \mathcal{H}(t,x_t,\bar{u}_t,\bar\gamma_t,\bar\eta_t)
    =
    \sup_{\bar\gamma,\bar\eta}\inf_{\bar{u}}
    \mathcal{H}(t,x_t,\bar{u}_t,\bar\gamma_t,\bar\eta_t)
\end{align}
holds in \eqref{eq:DPP:Vquadform2}. Therefore, we can apply the saddle point condition
\begin{align}
    \mathcal{H}(t,x_t,u^*_t,\bar\gamma_t,\bar\eta_t)
\leq \mathcal{H}(t,x_t,u^*_t,\gamma^*_t,\eta^*_t)
\leq \mathcal{H}(t,x_t,\bar{u}_t,\gamma^*_t,\eta^*_t)
\end{align} 
to solve the game.

Applying the definition of the Hamiltonian at \eqref{eq:Hamiltonian}, the saddle point condition implies that
\begin{align}\label{Ham}
    & \theta \left[ 
        x_t' M x_t  
        + \text{tr} \left( \Xi_t N_t \right) + \left( u^*_t + \bar\eta_t \right)'N_t\left( u^*_t + \bar\eta_t \right)
        + \left( u^*_t + \bar\eta_t \right)' Q x_t 
        + x_t' m
    + \left( u^*_t + \bar\eta_t \right)' n \right]
                        \\
    &
    -\frac{1}{2} \left( \bar\gamma_t' \Lambda^{-1} \bar\gamma_t + \bar\eta_t' \Xi_t^{-1} \bar\eta_t \right)
    + \mathbf{E}_{t,x_t}^{\bar\gamma_t, \bar\eta_t} \left[V_{t+1}(x_{t+1}) \right]                        \nonumber\\
    \leq& \theta \left[ 
        x_t' M x_t  
        + \text{tr} \left( \Xi_t N_t \right) + \left( u^*_t + \eta^*_t \right)'N_t\left( u^*_t + \eta^*_t \right)
        + \left( u^*_t + \eta^*_t \right)' Q x_t 
        + x_t' m
        + \left( u^*_t + \eta^*_t \right)' n \right]
                        \nonumber\\
    &
    -\frac{1}{2} \left( (\gamma^*_t)' \Lambda^{-1} \gamma^*_t + (\eta^*_t)' \Xi_t^{-1} \eta^*_t\right)
    + \mathbf{E}_{t,x_t}^{\gamma^*, \eta^* } \left[V_{t+1}(x_{t+1}) \right]
                        \nonumber\\
    \leq& \theta \left[ 
        x_t' M x_t  
        + \text{tr} \left( \Xi_t N_t \right) + \left( \bar{u}_t + \eta^*_t \right)'N_t\left( \bar{u}_t + \eta^*_t \right)
        + \left( \bar{u}_t +\eta^*_t \right)' Q x_t 
        + x_t' m
+ \left( \bar{u}_t + \eta^*_t \right)' n \right]
                        \nonumber\\
    &
    -\frac{1}{2} \left( (\gamma^*_t)' \Lambda^{-1} \gamma^*_t + (\eta^*)' \Xi_t^{-1} \eta^*_t \right)
    + \mathbf{E}_{t,x_t}^{\gamma^*_t, \eta^*_t} \left[V_{t+1}(x_{t+1}) \right].
                        \nonumber
\end{align}

Using again \eqref{eq:V:quadform:xt}  as well as the definition of the function $F(\cdot)$ at \eqref{eq:auxfunc:F:shorthand} together with that of the shorthand notations \eqref{short} and \eqref{short2}, we find that the previous relation \eqref{Ham} simplifies to
\begin{align} \label{eq:saddle:1}
     F(u^*_t,\bar\gamma_t,\bar\eta_t) 
     \leq
     F(u^*_t,\gamma^*_t,\eta^*_t) 
     \leq 
     F(\bar{u}_t,\gamma^*_t,\eta^*_t),   
\end{align}
where $F$ is defined above at \eqref{eq:auxfunc:F:shorthand}. Hence, the search for a saddle point for $V_t(x_t)$ reduces to the search for a saddle point for $F(\bar u,\bar\gamma,\bar\eta)$.  

The next step is to identify the unique stationary point $(u^*, \gamma^*, \eta^*)$ of the quadratic function $F(\bar u_t,\bar\gamma,\bar\eta)$. By the first order condition, $u^*$ satisfies
\begin{align}\label{eq:ustar:1}
    &\frac{\partial F}{\partial \bar{u}}(\bar{u}, \bar\gamma, \bar\eta)\Big|_{\bar{u} = u^*}
    = 0                                                                             \nonumber\\
    \Leftrightarrow& = \mathfrak{B}^{(1)}_{t+1} u^*
        + \mathfrak{A}^{(1)}_{t+1} x_t  
        + \mathfrak{a}^{(1)}_{t+1} 
        + \mathfrak{C}_{t+1} \bar\gamma
        + \mathfrak{B}^{(1)}_{t+1} \bar\eta
        = 0
                                    \nonumber\\
\Leftrightarrow& 
    u^*
    = -\left(\mathfrak{B}^{(1)}_{t+1}\right)^{-1}
    \left[ 
        \mathfrak{A}^{(1)}_{t+1} x_t  
        + \mathfrak{a}^{(1)}_{t+1} 
        + \mathfrak{C}_{t+1} \bar\gamma
    \right]
    - \bar\eta
\end{align}

Applying again the first order condition,  $\gamma^*$ satisfies

\begin{align}\label{eq:gammastar:1}
    &\frac{\partial F}{\partial \bar \gamma}(\bar{u}, \bar\gamma, \bar\eta)\Big|_{\bar\gamma = \gamma^*}
        = \mathfrak{B}^{(2)}_{t+1} \bar\gamma
        + \mathfrak{A}^{(2)}_{t+1} x_t
        + \mathfrak{a}^{(2)}_{t+1}
        + \mathfrak{C}_{t+1}' (\bar{u} +\bar\eta)   
        = 0
                                    \nonumber\\
\Leftrightarrow& 
    \gamma^* =  -\left( \mathfrak{B}^{(2)}_{t+1} \right)^{-1} \left[
            \mathfrak{A}^{(2)}_{t+1} x_t
            + \mathfrak{a}^{(2)}_{t+1}
        + \mathfrak{C}_{t+1}' (\bar{u} +\bar\eta)
        \right],
\end{align} 
and $\eta^*$ satisfies
\begin{align}\label{eq:etastar:1}
    &\frac{\partial F}{\partial \bar\eta}(\bar{u}, \bar\gamma, \bar\eta)\Big|_{\bar\eta = \eta^*}
    = 0
                                    \nonumber\\
\Leftrightarrow&     
        \mathfrak{B}^{(3)}_{t+1}\eta^*
        + \mathfrak{A}^{(1)}_{t+1} x_t 
        + \mathfrak{a}^{(1)}_{t+1} 
        + (\mathfrak{B}^{(1)}_{t+1})' \bar{u}
        + \mathfrak{C}_{t+1} \bar\gamma   
        = 0
                                    \nonumber\\
\Leftrightarrow& 
    \eta^* = - \left(\mathfrak{B}^{(3)}_{t+1}\right)^{-1} \left[
        \mathfrak{A}^{(1)}_{t+1} x_t 
        + \mathfrak{a}^{(1)}_{t+1} 
        + (\mathfrak{B}^{(1)}_{t+1})' \bar{u}
        + \mathfrak{C}_{t+1} \bar\gamma
    \right]
\end{align}
From here, using the definition of $\mathfrak{B}^{(3)}_{t+1}$, the expressions for $u^*$ and $\eta^*$ obtained above as well as the symmetry of $\mathfrak{B}^{(1)}_{t+1}$, we obtain
\begin{align}\label{eq:etastar:2}
    &   \mathfrak{B}^{(3)}_{t+1}\eta^*
        + \mathfrak{A}^{(1)}_{t+1} x_t 
        + \mathfrak{a}^{(1)}_{t+1} 
        + (\mathfrak{B}^{(1)}_{t+1})' u^*
        + \mathfrak{C}_{t+1} \gamma^*
        = 0
                                \nonumber\\
    \Leftrightarrow
    & \left(-\Xi_t^{-1} + \mathfrak{B}^{(1)}_{t+1}\right)\eta^*
        +  \mathfrak{A}^{(1)}_{t+1} x_t 
        + \mathfrak{a}^{(1)}_{t+1} 
        + \mathfrak{B}^{(1)}_{t+1} u^*
        + \mathfrak{C}_{t+1} \gamma^* 
        = 0
                                \nonumber\\
        \eta^* = 0
\end{align}
Substituting into \eqref{eq:ustar:1} and \eqref{eq:gammastar:1}, and solving the system \eqref{eq:ustar:1}-\eqref{eq:gammastar:1} for $u^*$ and $\gamma^*$, we obtain \eqref{eq:star:2}.

From the expressions \eqref{eq:ustar:1} for $u^*$, \eqref{eq:gammastar:1}for $\gamma^*$, and \eqref{eq:etastar:2} for $\eta^*$, we conclude the proof of the second part of the Proposition.
\end{proof} 

\begin{remark}
Recalling equation \eqref{eq:ustar:1} , we have 
\begin{align}\label{eq:u&eta}
u^* + \bar\eta
= -\left(\mathfrak{B}^{(1)}_{t+1}\right)^{-1}\left[ \mathfrak{A}^{(1)}_{t+1} x_t  + \mathfrak{a}^{(1)}_{t+1} + \mathfrak{C}_{t+1} \gamma^*\right],
\end{align}
which is the mean of the distribution of exploratory policies under $\mathbb{P}^{\bar\gamma,\bar\eta}$. By Proposition \ref{P1} we have $\eta^*=0$, so the optimal exploration is unbiased. Hence, we may write

\begin{align}\label{eq:star:3}
    u^*
    = -\left(\mathfrak{B}^{(1)}_{t+1}\right)^{-1}
    \left[
        \mathfrak{A}^{(1)}_{t+1} x_t  
        + \mathfrak{a}^{(1)}_{t+1}
    \right]
    -\left(\mathfrak{B}^{(1)}_{t+1}\right)^{-1} \mathfrak{C}_{t+1}\gamma^*.
\end{align}
This relation shows that we can decompose $u^*$ into the unconstrained asset allocation 
$$-\left(\mathfrak{B}^{(1)}_{t+1}\right)^{-1}
    \left[
        \mathfrak{A}^{(1)}_{t+1} x_t  
        + \mathfrak{a}^{(1)}_{t+1}
\right]$$ 
and a penalty term 
$$-\left(\mathfrak{B}^{(1)}_{t+1}\right)^{-1} \mathfrak{C}_{t+1}\gamma^*$$ 
induced by the free energy-entropy duality penalization and the choice of an optimal measure $\mathbb{P}^{\gamma^*,\eta^*}$.

\end{remark}

Next, we have 

\begin{proposition}\label{P2}
At the stationary point $(u_t^*,\gamma_t^*,\eta_t^*)$ the function $F(\bar u,\bar\gamma,\bar\eta)$ from \eqref{eq:auxfunc:F:shorthand} 
has the quadratic representation

\begin{align}\label{eq:auxfunc:F:optim:final}
    F(u_t^*, \gamma_t^*, \eta_t^*)  
    =& \frac{1}{2} x_t' \mathfrak{Q}_{t+1} x_t 
        + x_t' \mathfrak{q}_{t+1} x_t
        + \mathfrak{k_{t+1}},      
\end{align}
with
\begin{align}
    \mathfrak{Q}_{t+1}
    =&  - 
        \left(\mathfrak{A}^{(1)}_{t+1}\right)'
            (\mathfrak{B}^{(1)}_{t+1})^{-1} \mathfrak{A}^{(1)}_{t+1}
        - 2 \left(\mathfrak{A}^{(1)}_{t+1}\right)' 
            (\mathfrak{B}^{(1)}_{t+1})^{-1} 
            \mathfrak{C}_{t+1}(\mathfrak{G}_{t+1})^{-1} \mathfrak{A}^{(2)}_{t+1}  
                                \\
    &    +  \left(\mathfrak{A}^{(1)}_{t+1}\right)' 
            (\mathfrak{B}^{(1)}_{t+1})^{-1} 
            \mathfrak{C}_{t+1}(\mathfrak{G}_{t+1})^{-1}
            \mathfrak{C}_{t+1}' \left(\mathfrak{B}^{(1)}_{t+1}\right)^{-1} 
            \mathfrak{A}^{(1)}_{t+1} 
        +   \left(\mathfrak{A}^{(2)}_{t+1}\right)'
            (\mathfrak{G}_{t+1})^{-1}
            \mathfrak{A}^{(2)}_{t+1} 
                      \nonumber      
\end{align}     
\begin{align}
    \mathfrak{q}_{t+1}
    =&  - \left(\mathfrak{A}^{(1)}_{t+1}\right)'
            (\mathfrak{B}^{(1)}_{t+1})^{-1} 
            \mathfrak{a}^{(1)}_{t+1}
        - \left(\mathfrak{A}^{(1)}_{t+1}\right)'
            (\mathfrak{B}^{(1)}_{t+1})^{-1} 
            \mathfrak{C}_{t+1}(\mathfrak{G}_{t+1})^{-1} \mathfrak{a}^{(2)}_{t+1}
                                \\
    &   + \left(\mathfrak{A}^{(1)}_{t+1}\right)'
            (\mathfrak{B}^{(1)}_{t+1})^{-1} 
            \mathfrak{C}_{t+1}(\mathfrak{G}_{t+1})^{-1} \mathfrak{C}_{t+1}' \left(\mathfrak{B}^{(1)}_{t+1}\right)^{-1} \mathfrak{a}^{(1)}_{t+1}  
                                \nonumber\\
    &   - \left(\mathfrak{A}^{(2)}_{t+1}\right)'         
            (\mathfrak{G}_{t+1})^{-1} \mathfrak{C}_{t+1}' 
            (\mathfrak{B}^{(1)}_{t+1})^{-1} 
            \mathfrak{a}^{(1)}_{t+1}  
        + \left(\mathfrak{A}^{(2)}_{t+1}\right)'         
            (\mathfrak{G}_{t+1})^{-1} 
            \mathfrak{a}^{(2)}_{t+1}  
                                \nonumber
\end{align}
\begin{align}
\mathfrak{k}_{t+1}
    =&  - \frac{1}{2} 
            \left(\mathfrak{a}^{(1)}_{t+1}\right)' 
            (\mathfrak{B}^{(1)}_{t+1})^{-1} 
            \mathfrak{a}^{(1)}_{t+1}
        -   \left(\mathfrak{a}^{(1)}_{t+1}\right)' 
            (\mathfrak{B}^{(1)}_{t+1})^{-1} 
            \mathfrak{C}_{t+1}(\mathfrak{G}_{t+1})^{-1} \mathfrak{a}^{(2)}_{t+1}
                            \\
    &   +\frac{1}{2} 
            \left(\mathfrak{a}^{(1)}_{t+1}\right)' 
            (\mathfrak{B}^{(1)}_{t+1})^{-1} 
            \mathfrak{C}_{t+1}(\mathfrak{G}_{t+1})^{-1} \mathfrak{C}_{t+1}' \left(\mathfrak{B}^{(1)}_{t+1}\right)^{-1} \mathfrak{a}^{(1)}_{t+1} 
        + \frac{1}{2}       
            \left(\mathfrak{a}^{(2)}_{t+1}\right)' 
            (\mathfrak{G}_{t+1})^{-1}
            \mathfrak{a}^{(2)}_{t+1} 
                            \nonumber
\end{align}

\end{proposition}

\begin{proof}
Replacing the values of $u^*,\gamma^*,\eta^*$ from \eqref{eq:ustar:1}, \eqref{eq:gammastar:1}, \eqref{eq:etastar:2} into $F(\bar u,\bar\gamma,\bar\eta)$ and using locally the shorthand notations
\begin{align}
    K^{(1)}_{t+1}
    =& \mathfrak{A}^{(1)}_{t+1} 
            + \mathfrak{C}_{t+1}(\mathfrak{G}_{t+1})^{-1} K^{(2)}_{t+1}&K^{(2)}_{t+1}
    :=& \mathfrak{A}^{(2)}_{t+1} -\mathfrak{C}_{t+1}'\left(\mathfrak{B}^{(1)}_{t+1}\right)^{-1}\mathfrak{A}^{(1)}_{t+1}
                                \\      
    k^{(1)}_{t+1} 
    =& \mathfrak{a}^{(1)}_{t+1} 
        + \mathfrak{C}_{t+1}(\mathfrak{G}_{t+1})^{-1} k^{(2)}_{t+1}&k^{(2)}_{t+1}
    :=& \mathfrak{a}^{(2)}_{t+1}
    -\mathfrak{C}_{t+1}'\left(\mathfrak{B}^{(1)}_{t+1}\right)^{-1}\mathfrak{a}^{(1)}_{t+1} 
                                \nonumber
\end{align}
we obtain
\begin{align}
    & F(u^*, \gamma^*, \eta^*)    
                                \\
    =& \frac{1}{2} x_t'\left(K^{(1)}_{t+1}\right)' 
            (\mathfrak{B}^{(1)}_{t+1})^{-1} 
            K^{(1)}_{t+1} x_t 
        + x_t'\left(K^{(1)}_{t+1}\right)'
            (\mathfrak{B}^{(1)}_{t+1})^{-1}  k^{(1)}_{t+1}  
                                \nonumber\\
    &   + \frac{1}{2} \left(k^{(1)}_{t+1}
            \right)'
            (\mathfrak{B}^{(1)}_{t+1})^{-1} 
            k^{(1)}_{t+1}
        - x_t'\left(K^{(1)}_{t+1}\right)'
            (\mathfrak{B}^{(1)}_{t+1})^{-1} 
            \mathfrak{A}^{(1)}_{t+1} x_t    
                            \nonumber\\
    &   - x_t' \left(K^{(1)}_{t+1}\right)'
            (\mathfrak{B}^{(1)}_{t+1})^{-1} 
            \mathfrak{a}^{(1)}_{t+1} 
        - \left(k^{(1)}_{t+1}\right)'
            (\mathfrak{B}^{(1)}_{t+1})^{-1} 
            \mathfrak{A}^{(1)}_{t+1} x_t
                            \nonumber\\
    &   - \left(k^{(1)}_{t+1}\right)'
            (\mathfrak{B}^{(1)}_{t+1})^{-1} 
            \mathfrak{a}^{(1)}_{t+1} 
                                \nonumber\\
    &   + \frac{1}{2} x_t'
            \left(K^{(2)}_{t+1}\right)' 
            (\mathfrak{G}_{t+1})^{-1}
            \mathfrak{B}^{(2)}_{t+1} (\mathfrak{G}_{t+1})^{-1} 
            K^{(2)}_{t+1} x_t 
                            \nonumber\\
    &   + x_t'
            \left(K^{(2)}_{t+1}\right)' 
            (\mathfrak{G}_{t+1})^{-1} 
            \mathfrak{B}^{(2)}_{t+1} (\mathfrak{G}_{t+1})^{-1} 
            k^{(2)}_{t+1} 
                                \nonumber\\
    &+ \frac{1}{2} \left(k^{(2)}_{t+1}\right)' 
            (\mathfrak{G}_{t+1})^{-1}
            \mathfrak{B}^{(2)}_{t+1} (\mathfrak{G}_{t+1})^{-1} 
            k^{(2)}_{t+1} 
                                \nonumber\\
    &+ x_t'\left(K^{(2)}_{t+1}\right)' 
            (\mathfrak{G}_{t+1})^{-1}
            \mathfrak{A}^{(2)}_{t+1} x_t
        + x_t'\left(K^{(2)}_{t+1}\right)' 
            (\mathfrak{G}_{t+1})^{-1} \mathfrak{a}^{(2)}_{t+1}
                                \nonumber\\
    &+ \left(k^{(2)}_{t+1}\right)' 
            (\mathfrak{G}_{t+1})^{-1}
            \mathfrak{A}^{(2)}_{t+1} x_t
        + \left(k^{(2)}_{t+1}\right)' 
            (\mathfrak{G}_{t+1})^{-1}
            \mathfrak{a}^{(2)}_{t+1}
                            \nonumber\\
    &- x_t'\left(K^{(1)}_{t+1}\right)'
            \left(\mathfrak{B}^{(1)}_{t+1}\right)^{-1}
            \mathfrak{C}_{t+1} \left(\mathfrak{G}_{t+1}\right)^{-1} 
            K^{(2)}_{t+1} x_t 
                                \nonumber\\
    &- x_t'\left(K^{(1)}_{t+1}\right)'
            \left(\mathfrak{B}^{(1)}_{t+1}\right)^{-1}
            \mathfrak{C}_{t+1} \left(\mathfrak{G}_{t+1}\right)^{-1} 
            k^{(2)}_{t+1} 
                            \nonumber\\
    &- \left(k^{(1)}_{t+1}\right)'
            \left(\mathfrak{B}^{(1)}_{t+1}\right)^{-1}
            \mathfrak{C}_{t+1} \left(\mathfrak{G}_{t+1}\right)^{-1} 
            K^{(2)}_{t+1} x_t 
                                \nonumber\\
    &- \left( k^{(1)}_{t+1} \right)'
            \left(\mathfrak{B}^{(1)}_{t+1}\right)^{-1}
            \mathfrak{C}_{t+1} \left(\mathfrak{G}_{t+1}\right)^{-1} 
            k^{(2)}_{t+1} 
                                \nonumber
\end{align} 
Representing this expression as a quadratic form in $x_t$, we obtain, after some tedious calculations, the relation \eqref{eq:auxfunc:F:optim:final} with the coefficients as indicated in the statement of the proposition.

\end{proof}

We can now state our main

\begin{theorem}\label{theo:main:recursions}
The value function $V$ has, as mentioned in \eqref{eq:V:quadform0}, a quadratic form of the type
\begin{align}\label{eq:V:quadform}
    V_t(x_t) = \frac{1}{2}x_t' P_t x_t + x_t' p_t + r_t,
\end{align}
where $P_t$, $p_t$, and $r_t$ are deterministic and satisfy the following backward recursions
\begin{enumerate}[(i)]   
\item
\begin{align}\label{eq:recursion:P}
    P_{t} 
    =&  - \left(\mathfrak{A}^{(1)}_{t+1}\right)'
            (\mathfrak{B}^{(1)}_{t+1})^{-1} \mathfrak{A}^{(1)}_{t+1}
                                \\
    &   - 2 \left(\mathfrak{A}^{(1)}_{t+1}\right)' 
            (\mathfrak{B}^{(1)}_{t+1})^{-1} 
            \mathfrak{C}_{t+1}(\mathfrak{G}_{t+1})^{-1} \mathfrak{A}^{(2)}_{t+1}  
                                \nonumber\\
    &   + \left(\mathfrak{A}^{(1)}_{t+1}\right)' 
            (\mathfrak{B}^{(1)}_{t+1})^{-1} 
            \mathfrak{C}_{t+1}(\mathfrak{G}_{t+1})^{-1}
            \mathfrak{C}_{t+1}' \left(\mathfrak{B}^{(1)}_{t+1}\right)^{-1} 
            \mathfrak{A}^{(1)}_{t+1}  
                                \nonumber\\
    &   +   \left(\mathfrak{A}^{(2)}_{t+1}\right)'
            (\mathfrak{G}_{t+1})^{-1}
            \mathfrak{A}^{(2)}_{t+1}
        + 2 \left(
            \theta M + \frac{1}{2} A'P_{t+1}A 
        \right),
                                \nonumber\\
    P_T &= M_T,
                                \nonumber
\end{align}

\item
\begin{align}\label{eq:recursion:p}
    p_{t} 
    =&  - \left(\mathfrak{A}^{(1)}_{t+1}\right)'
            (\mathfrak{B}^{(1)}_{t+1})^{-1} 
            \mathfrak{a}^{(1)}_{t+1}
        - \left(\mathfrak{A}^{(1)}_{t+1}\right)'
            (\mathfrak{B}^{(1)}_{t+1})^{-1} 
            \mathfrak{C}_{t+1}(\mathfrak{G}_{t+1})^{-1} \mathfrak{a}^{(2)}_{t+1}
                                \\
    &   + \left(\mathfrak{A}^{(1)}_{t+1}\right)'
            (\mathfrak{B}^{(1)}_{t+1})^{-1} 
            \mathfrak{C}_{t+1}(\mathfrak{G}_{t+1})^{-1} \mathfrak{C}_{t+1}' \left(\mathfrak{B}^{(1)}_{t+1}\right)^{-1} \mathfrak{a}^{(1)}_{t+1}  
                                \nonumber\\
    &   - \left(\mathfrak{A}^{(2)}_{t+1}\right)'
            (\mathfrak{G}_{t+1})^{-1} \mathfrak{C}_{t+1}' 
            (\mathfrak{B}^{(1)}_{t+1})^{-1} 
            \mathfrak{a}^{(1)}_{t+1}  
        + \left(\mathfrak{A}^{(2)}_{t+1}\right)'
            (\mathfrak{G}_{t+1})^{-1} 
            \mathfrak{a}^{(2)}_{t+1} 
                                \nonumber\\
    &   + A' P_{t+1} a 
        + \theta m 
        + A' p_{t+1},
                                \nonumber\\
        p_T =& m_T,             \nonumber
\end{align}

\item
\begin{align}\label{eq:recursion:k}
    r_{t} 
    =&  - \frac{1}{2} 
            \left(\mathfrak{a}^{(1)}_{t+1}\right)' 
            (\mathfrak{B}^{(1)}_{t+1})^{-1} 
            \mathfrak{a}^{(1)}_{t+1}
                            \\
    &   -   \left(\mathfrak{a}^{(1)}_{t+1}\right)' 
            (\mathfrak{B}^{(1)}_{t+1})^{-1} 
            \mathfrak{C}_{t+1}(\mathfrak{G}_{t+1})^{-1} \mathfrak{a}^{(2)}_{t+1}
                            \nonumber\\
    &   +\frac{1}{2} 
        \left(\mathfrak{a}^{(1)}_{t+1}\right)' 
            (\mathfrak{B}^{(1)}_{t+1})^{-1} 
            \mathfrak{C}_{t+1}(\mathfrak{G}_{t+1})^{-1} \mathfrak{C}_{t+1}' \left(\mathfrak{B}^{(1)}_{t+1}\right)^{-1} \mathfrak{a}^{(1)}_{t+1} 
                            \nonumber\\
    &   + \frac{1}{2}       
            \left(\mathfrak{a}^{(2)}_{t+1}\right)' 
            (\mathfrak{G}_{t+1})^{-1}
            \mathfrak{a}^{(2)}_{t+1}
                            \nonumber\\
    &   + \frac{1}{2} \tr\left( B'\Xi_t B P_{t+1} \right)
        + \frac{1}{2} \tr\left( \Lambda_t  P_{t+1} \right)
        + r_{t+1} 
                            \nonumber\\
    &   + \frac{1}{2} a'P_{t+1}a
        + a' p_{t+1}
        + \theta \tr \left( \Xi_t N_t \right),
                            \nonumber\\
    & r_T = 0.              \nonumber
\end{align}
\end{enumerate}

Furthermore, the candidate optimal controls are given by the stationary point $(u^*,\gamma^*,\eta^*)$ of $F(\bar u,\bar\gamma,\bar\eta)$ according to Proposition \ref{P1} (formulas \eqref{eq:ustar:1},\eqref{eq:gammastar:1}, \eqref{eq:etastar:2})

\end{theorem}

\begin{proof} We need to prove only the statement in the first part of the theorem concerning the expression for $V_t(x_t)$. From Propositions \ref{P1}  and \ref{P2}, we obtain
\begin{align}\label{eq:DPP:Vquadform:2}
    & V_t(x_t)
                                    \\
    =& F(u_t^*,\gamma_t^*,\eta_t^*)
        +   x_t' \left(\theta M + \frac{1}{2} A 'P_{t+1}A \right)x_t
        + x_t'\left(
            A' P_{t+1} a + \theta m + A' p_{t+1}
          \right)
                                    \nonumber\\
    &
        + \frac{1}{2} \tr\left( B'\Xi_t B P_{t+1} \right)
        + \frac{1}{2} \tr\left( \Lambda_t  P_{t+1} \right)
        + r_{t+1} 
        + \frac{1}{2} a'P_{t+1}a
        + a' p_{t+1}
        + \theta \tr \left( \Xi_t N_t \right)
                                    \nonumber\\
    =& \frac{1}{2} x_t' \mathfrak{Q}_{t+1} x_t 
        + x_t' \mathfrak{q}_{t+1} x_t
        + \mathfrak{k_{t+1}}
        +   x_t' \left(\theta M + \frac{1}{2} A 'P_{t+1}A \right)x_t
        + x_t'\left(
            A' P_{t+1} a + \theta m + A' p_{t+1}
          \right)
                                    \nonumber\\
    &
        + \frac{1}{2} \tr\left( B'\Xi_t B P_{t+1} \right)
        + \frac{1}{2} \tr\left( \Lambda_t  P_{t+1} \right)
        + r_{t+1} 
        + \frac{1}{2} a'P_{t+1}a
        + a' p_{t+1}
        + \theta \tr \left( \Xi_t N_t \right).
                                    \nonumber
\end{align}
which allows to conclude the proof.

\end{proof}

\begin{corollary}
    We have
    \begin{align}\label{eq:supI}
        \inf_{\bar{u}} I(\bar{u};T,\theta) = \exp\left\{\frac{1}{2}x_t' P_t x_t + x_t' p_t + r_t\right\}
    \end{align}
    and 
    \begin{align}\label{eq:supJ}
        \sup_{\bar{u}} J(\bar{u};T,\theta) = -\frac{1}{\theta}\left( \frac{1}{2}x_t' P_t x_t + x_t' p_t + r_t\right)
    \end{align}
    with $P_t$, $p_t$, and $r_t$ given by \eqref{eq:recursion:P},\eqref{eq:recursion:p}, and \eqref{eq:recursion:k}, and optimal controls given by the stationary point $(u^*,\gamma^*,\eta^*)$ according to Proposition \ref{P1}.
\end{corollary}

\begin{proof}
Equation \eqref{eq:supI}  follows immediately from \eqref{eq:EEDuality:inf} (see also \eqref{criterion1}), that is
\begin{align}
    &
    \inf_{\bar{u}} I(u;T,\theta) 
    = 
    \exp\left\{ \inf_{\bar{u}} \sup_{\bar\gamma,\bar\eta}
    \mathbf{E}^{\bar\gamma, \bar\eta} \left[ 
        \theta G_T 
        -\frac{1}{2} \sum_{t=0}^{T-1} \left( \bar\gamma_t' \Lambda^ {-1}_t \bar\gamma_t + \bar\eta_t' \Xi_t^{-1} \bar\eta_t\right)
    \right] \right\},
\end{align}
and \eqref{eq:supJ} follows from the definition of $I$ at \eqref{eq:criterion:I}.
\end{proof}

\subsection{Sufficient conditions for the existence of a saddle point}\label{sec:saddleconditions}

The first-order conditions have led us to identify in Proposition \ref{P1} the unique stationary point $(u^*,\eta^*,\gamma*)$ of the quadratic function $F(\bar u,\bar\gamma,\bar\eta)$. We next look for conditions that ensure that $(u^*,\eta^*,\gamma*)$ is actually a saddle point of $F(\bar u,\bar\gamma,\bar\eta)$ in the sense of \eqref{eq:saddle:1}. For this, we shall involve the second-order conditions. To begin with, notice that the Hessian for our game problem is
\begin{align}
H^{\bar{u},\bar\gamma,\bar\eta}_t = \begin{pmatrix}
    \frac{\partial^2 F}{\partial \bar{u}^2}
    & \frac{\partial^2 F}{\partial \bar{u} \partial \bar\gamma}
    & \frac{\partial^2 F}{\partial \bar{u} \partial \bar\eta}
    \\
	\frac{\partial^2 F}{\partial \bar\gamma \partial \bar{u}}
    & \frac{\partial^2 F}{\partial \bar\gamma^2} 
	& \frac{\partial^2 F}{\partial\bar\gamma \partial\bar\eta} 
    \\
	\frac{\partial^2 F}{\partial \bar\eta \partial \bar{u}}
    &\frac{\partial^2 F}{\partial \bar\eta \partial \bar\gamma}
	& \frac{\partial^2 F}{\partial \bar\eta^2}
\end{pmatrix}
=
\begin{pmatrix}
		\mathfrak{B}^{(1)}_{t+1} 
        & \mathfrak{C}_{t+1} 
        & \mathfrak{B}^{(1)}_{t+1} 
        \\
        \mathfrak{C}_{t+1}' 
        & \mathfrak{B}^{(2)}_{t+1}  
		& \mathfrak{C}_{t+1}' 
        \\ 
        \mathfrak{B}^{(1)}_{t+1}
		& \mathfrak{C}_{t+1}  
		& \mathfrak{B}^{(3)}_{t+1}
\end{pmatrix}
\end{align}
Expressing the maximizing controls as a block vector $\bar\nu = \begin{pmatrix} \bar\gamma' & \bar\eta'\end{pmatrix}'$, we rewrite the Hessian matrix as the following block saddle point matrix
\begin{align}
    H^{\bar{u},\bar\nu}_t = \begin{pmatrix}
    \frac{\partial^2 F}{\partial \bar{u}^2}
    & \frac{\partial^2 F}{\partial \bar{u} \partial \bar\nu}
    \\
	\frac{\partial^2 F}{\partial \bar\nu \partial \bar{u}}
    & H_t^{\bar\gamma,\bar\eta}
\end{pmatrix}
\end{align}
where $H_t^{\bar\gamma,\bar\eta}$ is itself a block matrix, specifically
\begin{align}
H^{\bar\gamma,\bar\eta}_t = \begin{pmatrix}
	\frac{\partial^2 F}{\partial \bar\gamma^2} 
	& \frac{\partial^2 F}{\partial\bar\gamma \partial\bar\eta} \\
	\frac{\partial^2 F}{\partial \bar\eta \partial \bar\gamma}
	& \frac{\partial^2 F}{\partial \bar\eta^2}
\end{pmatrix}
= 
\begin{pmatrix}
		\mathfrak{B}^{(2)}_{t+1}  
		& \mathfrak{C}_{t+1}' \\ 
		\mathfrak{C}_{t+1}  
		& \mathfrak{B}^{(3)}_{t+1}
\end{pmatrix}
\end{align}

We can now state 
\begin{proposition}\label{P3}  Under Assumption \ref{as:coeffs} the stationary point $(u_t^*,\eta_t^*,\gamma_t*)$ derived in Proposition \ref{P1} is a saddle point for $V_t(x_t)$ as expressed in \eqref{eq:DPP:Vquadform1}.
\end{proposition}

\begin{proof}  According to the proof of Proposition \ref{P1},  the search for a saddle point for $V_t(x_t)$ reduces to the search of a saddle point of $F(\bar u,\bar\gamma,\bar\eta)$ in the sense of \eqref{eq:saddle:1}. It means that, at the stationary point $(u^*,\eta^*,\gamma*)$, the function $F(\bar u,\bar\gamma,\bar\eta)$ must have a minimum in $\bar u$ for all values of $\bar\nu=(\bar\gamma',\bar\eta_t')'$. A sufficient condition for this to happen is that, at $(u^*,\eta^*,\gamma*)$, the function $F(\bar u,\bar\gamma,\bar\eta)$ is convex in $\bar u$.  Analogously, the function  $F(\bar u,\bar\gamma,\bar\eta)$ must have a maximum in $\nu^*=({\gamma^*}',{\eta^*}')'$ for all values of $\bar u$. A sufficient condition is that, at $(u^*,\eta^*,\gamma^*)$, the function $F(\bar u,\bar\gamma,\bar\eta)$ is concave in $\bar \nu$. We do not need further global convexity/concavity properties of $F(\bar u,\bar\gamma,\bar\eta)$. Therefore, we can apply the usual second-order sufficiency condition separately to $\bar u$ and $\bar\nu$, namely
\begin{equation}\label{suff:1}
\frac{\partial^2 F}{\partial \bar{u}^2}(\bar u_,\bar\nu)=\mathfrak{B}^{(1)}_{t+1}=B'P_{t+1}B+2\theta\,N>0,\quad \forall t\in[0,T]
\end{equation}
\begin{equation}\label{suff:2}
-H^{\bar\gamma,\bar\eta}_t=-\begin{pmatrix}
		\mathfrak{B}^{(2)}_{t+1}  
		& \mathfrak{C}_{t+1}' \\
		\mathfrak{C}_{t+1}  
		& \mathfrak{B}^{(3)}_{t+1}
\end{pmatrix}>0,
\end{equation}
which is guaranteed by Assumption \ref{as:coeffs}.
\end{proof}

\begin{remark}\label{Rsuff}
Notice that conditions such as \eqref{suff:1} are standard for both LQG and LEQG problems, and that condition \eqref{suff:2} results from the energy-entropy duality. According to Sylvester's criterion, we can express the negative definiteness condition \eqref{suff:2} as the following three simultaneous conditions
\begin{equation}\label{suffcon}
\left\{\begin{array}{l}
det\left(H^{\bar\gamma,\bar\eta}_t\right)>0 \\
 \frac{\partial^2 F}{\partial \bar\gamma^2}(u_t^*,\gamma_t^*,\eta_t^*)=\mathfrak{B}^{(2)}_{t+1}=P_{t+1}-\Lambda_t^{-1}<0,\\
\frac{\partial^2 F}{\partial \bar\eta^2}(u_t^*,\gamma_t^*,\eta_t^*)=\mathfrak{B}^{(3)}_{t+1}=-\Xi_t^{-1}+B'P_{t+1}B+2\theta\,N_t<0
\end{array}
\right.
.
    \end{equation}
    
The second condition in \eqref{suffcon} is the second-order condition we would have if we maximized for $\bar\gamma$ only, ignoring $\bar\eta$. As $\Lambda_t$ and $P_{t+1}$ are both deterministic, this condition can be checked in advance, even if the recursion for $P_{t}$ depends on $\Lambda_{t+1}$ via the terms $\mathfrak{B}^{(2)}_{t+1}$ and $\mathfrak{G}_{t+1}$, defined respectively at \eqref{short} and \eqref{short2}.  This second condition in \eqref{suffcon} is also reminiscent of the \emph{risk-resistance condition} found in the LEQG literature \citep{ShaijuPetersenFormulasLQR_LQ_LEQG2008}. This was to be expected. The energy-entropy duality expresses the original LEQG problem as an LQG problem. Hence, any condition necessary for the resolution of the LEQG problem should also apply in some way to the equivalent LQG game. Such a condition will manifest through the penalizing control $\gamma$, which is the driving term of the change of measure in the energy-entropy duality. 

Analogously, the last condition in \eqref{suffcon}  is the second-order condition if we maximize for $\bar\eta$ and ignore $\bar\gamma$. This condition sets a bound on the covariance of exploration $\Xi_t,$ based on the parameters of the problem, $\theta, N_t, B$, and the quadratic coefficient $P_{t+1}$. As $\Xi_t$ and $P_{t+1}$ are deterministic, an appropriate value can be selected in advance for $\Xi_t$.

\end{remark}

\section{Analysis and Conclusions}\label{S.4}

 \subsection{Analysis of the sufficiency conditions for the existence of a saddle point}\label{S.4.1}

Proposition \ref{P3} showed that the Assumptions \ref{as:coeffs} are sufficient for the controls $(u^*,\eta^*,\gamma*)$ from Proposition \ref{P1} to lead to a saddle point for $V_t(x_t)$ as expressed in \eqref{eq:DPP:Vquadform1}. Using Remark \ref{Rsuff}, we express these conditions in terms of the original model parameters in Proposition \ref{P4}. To simplify the technicalities and clarify the exposition, we consider the scalar case. Nevertheless, this special case allows us to analyze the impact of these sufficiency conditions when applying the solution obtained in section \ref{S.3}.

\begin{proposition}\label{P4}
Given scalar model parameters $A>0,B>0,M>0,\theta>0$ and choosing for simplicity of presentation $Q=m=n=0$, an explicit sufficient condition for the stationary point 
$(u^*,\eta^*,\gamma*)$ from Proposition \ref{P1} to be a saddle point for $V_t(x_t)$ with a $\min$ in $u^*$ and a joint $\max$ in $\nu^*=(\gamma^*,\eta^*)$ is given by the following three requirements:
\begin{enumerate}[i)]
\item $N_t > \max \, \left[0,-\frac{B^2P_{t+1}}{2\theta}\right] \quad \mbox{(binding only if $P_{t+1}<0$)}$;
\item $\Lambda_t^{-1}>P_{t+1}$;
\item $\Xi_t^{-1}>\max\left[2\theta\,N_t+B^2P_{t+1},2\theta\,N_t+\frac{\Lambda_t^{-1}B^2P_{t+1}}{P_{t+1}-\Lambda_t^{-1}}\right]
=
\begin{cases} 2\theta\,N_t+B^2P_{t+1} & \text{if } P_{t+1}>0 \\
2\theta\,N_t+\frac{\Lambda_t^{-1}B^2P_{t+1}}{P_{t+1}-\Lambda_t^{-1}}    & \text{if } P_{t+1}<0
\end{cases}$
\end{enumerate}

\end{proposition}

\begin{proof}
\mbox{}\\*
Item i) implies \eqref{suff:1} for $\theta>0$. Item ii) implies the second condition in \eqref{suffcon} while Item iii) implies the third condition in \eqref{suffcon}. Finally, Item iii) together with Item ii)  imply the first condition in \eqref{suffcon}.

\end{proof}

We can now propose the following

\noindent{\bf Procedure {(\color{Maroon} recursion)}:}
Given is a discrete time period $t=0,\cdots,T$
\begin{enumerate}[1.]
\item Initialize the procedure by choosing $M,M_T,Q,m,m_T,n$  and by selecting a value for the risk sensitivity parameter $\theta\>\>(\theta>0)$. 
For simplicity of presentation and in line with Proposition \ref{P4}, also choose $Q=m=n=0$.
\item Put $P_T=M_T, p_T=m_T, r_T=0$ 
\item Register the estimated values for $A$ and $B$ as resulting from the most recent estimation/exploration step.
\item Choose $N_{T-1}$ according to i) in Proposition \ref{P4}  and choose $\Lambda_{T-1}$ and $\Xi_{T-1}$ according to ii) and iii) in Proposition \ref{P4}.
\item Compute $P_{T-1},p_{T-1},r_{T-1}$ according to \eqref{eq:recursion:P},  \eqref{eq:recursion:p}, \eqref{eq:recursion:k} in Theorem \ref{theo:main:recursions} and repeat steps 4. and 5. successively for the times $T-1,T-2,\cdots,1,0$.
\item Run the system from $t=0$ to $t=T$ with the control $u^*$ given by \eqref{eq:star:2}  of Proposition  \ref{P1} and with an unbiased randomization noise $v_t$ given by a zero-mean Gaussian with variance $\Xi_t$. For a simulation, choose the most recent estimated values of A and B and a system noise $w_t\sim \mathcal{N}(\gamma_t^*, \Lambda_t)$.
\item Reestimate $A$ and $B$ and repeat the previous steps down to $t=0$.

\end{enumerate}

\begin{remark} Since the criterion $I(u; T,\theta)$ is positive, the value function $V_t(x_t)$ will also be positive, so the most likely sign for $P_t$ is positive. By i) in Proposition \ref{P4}, the control penalization coefficient $N_t$ can then be chosen arbitrarily as a positive matrix. If we run a real system, the system noise $w_t$ might have a covariance $\Lambda_t$ with $\Lambda_t^{-1}<P_{t+1}$. In such a case, we cannot be sure that our controls $u^*$ and $\gamma^*$ lead to a saddle point as required in \eqref{eq:criterion:I:Pbar} (see also \eqref{eq:DPP:Vquadform1}).  On the other hand, if we simulate the system, the system noise has to have a variance that, in the scalar case, is smaller for larger values of $P_t$. Furthermore, the variance $\Xi_t$ of the unbiased randomization noise has to be such that
$$\Xi_t^{-1}>2\theta\,N_t+\max\left[B^2P_{t+1},\frac{\Lambda_t^{-1}B^2P_{t+1}}{P_{t+1}-\Lambda_t^{-1}}\right],$$
i.e. the more we penalize the control and the larger we take the risk sensitivity parameter $\theta$, the less dispersed we have to choose the randomization that provides exploration. The bound also depends on the computed value of $P_{t+1}$ that, in turn, depends on the estimated values of $A$ and $B$.
\end{remark}

\subsection{Numerical Example}\label{S.4.2}

We provide here a simple numerical example for the scalar case considered in the previous section \ref{S.4.1}. In addition to providing illustrative numbers, one major interest of this example is that it shows that the setup of Proposition \ref{P4} is sufficient, but not necessary. 

Table \ref{tab:example:parameters} reports the model parameters, which are all constant. This choice of parameters departs from the setup of Proposition \ref{P4} in two ways. First, $A = -0.2 <0$, implying that $x_t$ has a tendency to revert to its long-term mean. Second, $Q=1 \neq 0$ to model an interaction between the state and control in the cost function\footnote{This model specification is consistent with the continuous-time risk-sensitive models arising in investment management \citep[see for instance][]{kuna02,DavisLleoBook2014,LleoRunggaldierSeparation24}}. Nevertheless, we will see below that Assumption \ref{as:coeffs} is satisfied, and so, by Proposition \ref{P3}, the game has a unique saddle point at $(u^*,\gamma^*,\eta^*)$.

\begin{table}[H]
    \centering
    \begin{tabular}{|c|c|c|c|c|c|}
       \hline
       \multicolumn{2}{|c|}{Global Parameters}  
       &  \multicolumn{2}{c}{System Dynamics}
       &  \multicolumn{2}{|c|}{Cost Function}\\
       \hline
       Parameter    & Value                      
       &  Parameter & Value
       &  Parameter & Value \\
       \hline
       $\theta$     & 1                         
       & $a$    & 0 
       & $M$    & 2 \\
       $T$      & 25  
       & $A$    & -0.2
       & $N$    & 2   \\
       &  
       & $B$    & 0.4 
       & $Q$    & 1 \\
            &  
       & $\Lambda$  &  0.15
       & $m$        & 0 \\
         &  
       & $\Xi$ & 0.15 
       & $n$        & 0 \\
         &  
       & $x_0$ & 1 
       & $M_T$ & 4      \\
                &  
       &  &  
       & $m_T$ & 0      \\
    \hline
    \end{tabular}
    \caption{Model parameters}
    \label{tab:example:parameters}
\end{table}

The solution of the game is displayed in Table \ref{tab:example:results} below, which shows the evolution of the state variable $x_t$ under both the $\mathbb{P}$ and $\mathbb{P}^{\gamma^*,\eta^*}$ measures, of the value function $V_t(x_t)$ and its parameters $P_t, p_t, r_t$, of the optimal controls $u^*_t$, $\gamma^*_t$, and $\eta^*_t$, and of the intermediate Fraktur parameters required to check Assumption \ref{as:coeffs}. As $\mathfrak{B}^{(1)}_t >0, \; t=1,\ldots,T$, Assumption \ref{as:coeffs}(i) is satisfied. Next, we check that Assumption \ref{as:coeffs}(ii), that is,
\begin{equation}
-H^{\bar\gamma,\bar\eta}_t
=-\begin{pmatrix}
		\mathfrak{B}^{(2)}_{t+1}  
		& \mathfrak{C}_{t+1}' \\
		\mathfrak{C}_{t+1}  
		& \mathfrak{B}^{(3)}_{t+1}
\end{pmatrix}>0,
\end{equation}
is also satisfied. To check this assumption, we use Sylvester's criterion. From Table \ref{tab:example:results}, we have $\mathfrak{B}^{(2)}_t < 0, \; t=0,\ldots, T$. We also see that the determinant $
\det \left( -H^{\bar\gamma,\bar\eta}_t \right)
= \mathfrak{B}^{(2)}_{t}\mathfrak{B}^{(3)}_{t}  
- \mathfrak{C}_{t}^2
$, is strictly positive for $t=0,\ldots,T$. Therefore, Assumption \ref{as:coeffs} is fully satisfied and the game has a unique saddle point at $(u^*_t,\gamma^*_t,\eta^*_t)$. Notice also that, despite the different setups, the sufficient conditions i) to iii) of Proposition \ref{P4} would be satisfied for this example.   
\begin{sidewaystable}
    \resizebox{\textheight}{!}{\begin{tabular}{ccccccccccccccc}
        \hline
    $t$
    & $x_t$ under $\mathbb{P}$
    & $x_t$ under $\mathbb{P}^{\gamma^*,\eta^*}$
    & $P_t$
    & $p_t$
    & $r_t$
    & $V_t(x_t)$
    & $u^*_t$
    & $\gamma^*_t$
    & $\eta^*_t$   

    & $\mathfrak{B}^{(1)}_t$
    & $\mathfrak{B}^{(2)}_t$
    & $\mathfrak{B}^{(3)}_t$
    & $\mathfrak{C}_t$
    & $\det \left( -H^{\bar\gamma,\bar\eta}_t \right)$
    \\    
        \hline
0   &   1.0000   &   1.0000   &   4.5811   &   0.0000   &   17.3880   &   19.6786   &   0.0270   &   -0.4156   &   -0.0000   &      &      &      &      &      \\
1   &   -0.2655   &   -0.6782   &   4.5811   &   0.0000   &   16.6894   &   17.7430   &   -0.0183   &   0.2818   &   0.0000   &   4.7330   &   -2.0855   &   -1.9337   &   1.8324   &   0.6749   \\
2   &   -0.1205   &   0.2024   &   4.5811   &   0.0000   &   15.9909   &   16.0847   &   0.0055   &   -0.0841   &   -0.0000   &   4.7330   &   -2.0855   &   -1.9337   &   1.8324   &   0.6749   \\
3   &   -0.0357   &   -0.1270   &   4.5811   &   0.0000   &   15.2923   &   15.3293   &   -0.0034   &   0.0528   &   0.0000   &   4.7330   &   -2.0855   &   -1.9337   &   1.8324   &   0.6749   \\
4   &   -0.0629   &   0.0018   &   4.5811   &   0.0000   &   14.5938   &   14.5938   &   0.0000   &   -0.0007   &   -0.0000   &   4.7330   &   -2.0855   &   -1.9337   &   1.8324   &   0.6749   \\
5   &   -0.0581   &   -0.1008   &   4.5811   &   0.0000   &   13.8952   &   13.9185   &   -0.0027   &   0.0419   &   0.0000   &   4.7330   &   -2.0855   &   -1.9337   &   1.8324   &   0.6749   \\
6   &   0.0225   &   0.1376   &   4.5811   &   0.0000   &   13.1967   &   13.2400   &   0.0037   &   -0.0572   &   -0.0000   &   4.7330   &   -2.0855   &   -1.9337   &   1.8324   &   0.6749   \\
7   &   0.3837   &   0.2704   &   4.5811   &   0.0000   &   12.4981   &   12.6655   &   0.0073   &   -0.1123   &   -0.0000   &   4.7330   &   -2.0855   &   -1.9337   &   1.8324   &   0.6749   \\
8   &   0.2583   &   0.1856   &   4.5811   &   0.0000   &   11.7995   &   11.8785   &   0.0050   &   -0.0771   &   -0.0000   &   4.7330   &   -2.0855   &   -1.9337   &   1.8324   &   0.6749   \\
9   &   -0.3427   &   -0.3417   &   4.5811   &   0.0000   &   11.1010   &   11.3685   &   -0.0092   &   0.1420   &   -0.0000   &   4.7330   &   -2.0855   &   -1.9337   &   1.8324   &   0.6749   \\
10   &   0.0247   &   0.2576   &   4.5811   &   0.0000   &   10.4024   &   10.5544   &   0.0070   &   -0.1070   &   0.0000   &   4.7330   &   -2.0855   &   -1.9337   &   1.8324   &   0.6749   \\
11   &   0.0176   &   -0.1256   &   4.5811   &   0.0000   &   9.7039   &   9.7400   &   -0.0034   &   0.0522   &   -0.0000   &   4.7330   &   -2.0855   &   -1.9337   &   1.8324   &   0.6749   \\
12   &   0.0367   &   0.0863   &   4.5811   &   0.0000   &   9.0053   &   9.0224   &   0.0023   &   -0.0359   &   -0.0000   &   4.7330   &   -2.0855   &   -1.9337   &   1.8324   &   0.6749   \\
13   &   0.0111   &   -0.1089   &   4.5811   &   0.0000   &   8.3067   &   8.3339   &   -0.0029   &   0.0453   &   0.0000   &   4.7330   &   -2.0855   &   -1.9337   &   1.8324   &   0.6749   \\
14   &   0.1961   &   0.3188   &   4.5811   &   0.0000   &   7.6082   &   7.8409   &   0.0086   &   -0.1325   &   0.0000   &   4.7330   &   -2.0855   &   -1.9337   &   1.8324   &   0.6749   \\
15   &   -0.0315   &   -0.1993   &   4.5811   &   0.0000   &   6.9096   &   7.0006   &   -0.0054   &   0.0828   &   -0.0000   &   4.7330   &   -2.0856   &   -1.9337   &   1.8324   &   0.6750   \\
16   &   -0.1065   &   0.0364   &   4.5811   &   0.0000   &   6.2111   &   6.2141   &   0.0010   &   -0.0151   &   0.0000   &   4.7330   &   -2.0856   &   -1.9337   &   1.8324   &   0.6751   \\
17   &   0.1562   &   0.1640   &   4.5810   &   0.0000   &   5.5125   &   5.5741   &   0.0044   &   -0.0681   &   -0.0000   &   4.7330   &   -2.0857   &   -1.9337   &   1.8324   &   0.6755   \\
18   &   -0.0318   &   -0.1949   &   4.5807   &   0.0000   &   4.8140   &   4.9010   &   -0.0052   &   0.0809   &   -0.0000   &   4.7329   &   -2.0860   &   -1.9338   &   1.8323   &   0.6765   \\
19   &   -0.0052   &   0.0639   &   4.5800   &   0.0000   &   4.1156   &   4.1249   &   0.0017   &   -0.0265   &   0.0000   &   4.7328   &   -2.0867   &   -1.9339   &   1.8320   &   0.6792   \\
20   &   -0.0517   &   -0.0946   &   4.5780   &   0.0000   &   3.4173   &   3.4378   &   -0.0025   &   0.0392   &   0.0000   &   4.7325   &   -2.0887   &   -1.9342   &   1.8312   &   0.6867   \\
21   &   0.1290   &   0.2099   &   4.5725   &   0.0000   &   2.7195   &   2.8202   &   0.0051   &   -0.0863   &   0.0000   &   4.7316   &   -2.0942   &   -1.9351   &   1.8290   &   0.7073   \\
22   &   0.0813   &   -0.0077   &   4.5573   &   0.0000   &   2.0230   &   2.0231   &   -0.0001   &   0.0031   &   -0.0000   &   4.7292   &   -2.1094   &   -1.9375   &   1.8229   &   0.7638   \\
23   &   0.2333   &   0.2933   &   4.5150   &   0.0000   &   1.3302   &   1.5244   &   0.0015   &   -0.1122   &   -0.0000   &   4.7224   &   -2.1516   &   -1.9443   &   1.8060   &   0.9217   \\
24   &   0.0781   &   -0.0031   &   4.3929   &   0.0000   &   0.6480   &   0.6480   &   0.0001   &   0.0010   &   0.0000   &   4.7029   &   -2.2738   &   -1.9638   &   1.7571   &   1.3778   \\
25   &   0.2215   &   0.2497   &   4.0000   &   0.0000   &   0.0000   &   0.1247   &      &      &      &   4.6400   &   -2.6667   &   -2.0267   &   1.6000   &   2.8444   \\  
         \hline
    \end{tabular}}
    \caption{One run of the numerical example}
    \label{tab:example:results}
\end{sidewaystable}

\begin{remark}
As expected, in Table \ref{tab:example:results}, $\eta^*$ is practically 0.
\end{remark}

\subsection{Integration with Policy Gradient and Actor–Critic}\label{S.4.3}

In Sections \ref{sec:setup} and \ref{S.3}, we introduced the LEQG control problem, expressed it as an equivalent LQG game using the duality between free energy and entropy, and derived the solution to this game. Section \ref{S.4.1} then analyzed sufficient conditions for the game to have a unique saddle point and proposed a procedure to estimate sequentially the model parameters and check that the sufficient requirements of Proposition \ref{P4} are satisfied.  

As an alternative, we could also have developed and implemented a policy gradient approach, such as an actor-critic algorithm, to solve the problem numerically using only realized sample trajectories of the state space $x_t$ and cost $G_T$. A full and rigorous treatment is beyond the scope of this paper and should be the focus of further research. We thus only sketch the main ideas here. 

Policy gradient methods have gained popularity in recent years following groundbreaking advances in the field of reinforcement learning and a string of highly successful applications, most strikingly to the game of Go. We refer the reader to Chapter 13 in the classic text by \citet{SuttonBarto2018}. 

All policy gradient methods follow the same essential steps. First, approximate the policy using a conveniently parametrized function. Second, compute the gradient of a given performance measure, usually a value function, with respect to the parameters of the function approximating the controls. Third, update the parameters by performing a gradient descent update if the aim is to minimize the performance measure, or a gradient ascent update if the aim is to maximize the performance measure. 

Actor-critic methods build on this approach to update both the policy and the value function. The \emph{actor} performs the three policy gradient steps to update the parameters of the policy approximation, at the current estimate for the value function, while the \emph{critic} performs three similar steps to update the parameters of an approximation of the value function, at the current estimate for the policy. 

To implement a simple policy gradient method in our context, a natural approach is to extend the approach by \citet{hamblyPolicyGradientMethods2021,hamblyPolicyGradientMethods2023} on policy gradient methods for linear quadratic regulators and linear-quadratic games. Their work captures the essential features of LQG control problems and games, namely the quadratic cost in the state and in the controls, and the linear dynamics of the state variable, and provides strong convergence results. However, it does not consider some additional features that are present in our model and are important for applications, in particular to finance, such as the interaction between the state and control in the cost function, first-order terms in the cost function, and the possibility of a mean reversion to a long-term mean different from 0 for the state. Hence, further work is necessary. 

For our LQG game, the optimal controls are affine in the state, and the value function is quadratic. We can therefore use the following parametrization\footnote{Note that this is a full parametrization. We could use the fact that $\eta^* = 0$ to substantially simplify the mathematical exposition and reduce the amount of numerical work.}:
\begin{align}
    u_t = \mathtt{D}_t x_t + \mathtt{d}_t
                        \qquad
    \bar\gamma_t = \mathtt{E}_t x_t + \mathtt{e}_t
                        \qquad
    \bar\eta_t = \mathtt{F}_t x_t + \mathtt{f}_t
\end{align} 
for $t=0,\ldots,T-1$, where $\mathtt{D}_t \in \mathbb{R}^{d_u \times d_x}$, $\mathtt{d}_t \in \mathbb{R}^{d_u}$, $\mathtt{E}_t \in \mathbb{R}^{d_x \times d_x}$, $\mathtt{e}_t \in \mathbb{R}^{d_x}$, $\mathtt{F}_t \in \mathbb{R}^{d_u \times d_x}$, $\mathtt{f}_t \in \mathbb{R}^{d_u}$. Define $\mathtt{\mathbf{D}} := \left(\mathtt{D}_0, \ldots, \mathtt{D}_{T-1} \right)$ with similar definitions for $\mathtt{\mathbf{d}}, \mathtt{\mathbf{E}}, \mathtt{\mathbf{e}}, \mathtt{\mathbf{F}}, \mathtt{\mathbf{d}}$, and let $\mathbf{K} = \left(\mathtt{\mathbf{D}},\mathtt{\mathbf{d}},\mathtt{\mathbf{E}},\mathtt{\mathbf{e}},\mathtt{\mathbf{F}},\mathtt{\mathbf{f}}\right)$. Furthermore, we take the critic as
\begin{align}
V_t(x)=x_t' \mathtt{P}_t x + x_t'\mathtt{p}_t + \mathtt{r}_t,
\end{align}
where $\mathtt{P}_t \in \mathbb{R}^{d_x \times d_x}, \mathtt{p}_t \in \mathbb{R}^{d_x}$, and $\mathtt{r}_t \in \mathbb{R}$.

Based on \eqref{eq:criterion:I:Pbar}, the objective function $C$, as a function of the policy parameters, is
\begin{align}
    & C(\mathtt{D},\mathtt{d},\mathtt{E},\mathtt{e}, \mathtt{F},\mathtt{f}) 
                                \\
    =& \mathbf{E}^{\bar\gamma, \bar\eta} \left[ 
        \theta G_T(\mathtt{D},\mathtt{d},\mathtt{E},\mathtt{e}) 
        -\frac{1}{2} \sum_{t=0}^{T-1} \left\{\left(\mathtt{E}_t x_t + \mathtt{e}_t\right)' \Lambda^ {-1}_t \left(\mathtt{E}_t x_t + \mathtt{e}_t\right) 
        \right.\right.
                            \nonumber\\
    & \left.\left.
        + \left(\mathtt{F}_t x_t + \mathtt{f}_t\right)' \Xi_t^{-1} \left(\mathtt{F}_t x_t + \mathtt{f}_t\right)\right\}
    \right]
                            \nonumber
\end{align}
With this notation, Hambly et al.'s so-called \emph{Natural Policy Gradient} updating rules (extended to include a zeroth order term) for a sequence of episodes $m=1,\ldots,M$ are: 
\begin{align}
    \mathbf K^{(m+1)}=\mathbf K^{(m)} \pm \delta_m \,\widehat{\nabla}_{\mathbf K} C\!\left(\mathbf K^{(m)}\right)\, \left(\widehat{\Sigma}^{(m)}_x\right)^{-1},           
\end{align}
for $t=0,1,\ldots,T-1$ and where $\delta$ is the step size. For the case of unknown model parameters, $\widehat{\nabla}_{\mathbf K} C$ is the empirical policy gradient, and $\Sigma^{\mathbf{K}}_t := \mathbf{E}\left[ x^{\mathbf{K}}_t \left(x^{\mathbf{K}}_t\right)'\right]$ is the empirical state covariance (serving as a natural preconditioner). The sign is negative (descent) for the minimizing controller parameters \((\mathtt D,\mathtt d)\) and positive (ascent) for the adversarial parameters $(\mathtt E,\mathtt e,\mathtt F,\mathtt f)$.

From there, the next step is to consider a full actor-critic approach, where the gradient update to the policy approximation combines with a gradient update to the value function approximation. The critic parameters \((\mathtt P_t,\mathtt p_t,\mathtt r_t)\) are updated on a faster timescale by Temporal Differencing. The key here is to establish that the joint gradient updates will converge to the optimal controls and the optimal value function.

Finally, we note that global convergence results exist for policy-gradient methods for LQR problems-- see \citep{fazelGlobalConvergencePolicy2018,hamblyPolicyGradientMethods2021} for control problems and \citet{hamblyPolicyGradientMethods2023} for linear-quadratic games. On the actor-critic front, \citet{alacaogluNaturalActorCriticFramework2022} obtained promising results for 2-player Markov games.

\subsection{Conclusions}\label{S.4.4}

We have seen that an LEQG problem with a randomized control emulating exploration can, via the energy-entropy duality, be reduced to a risk-neutral LQG problem with an additive penalization. This LQG problem takes the form of a stochastic game where one player is the original controller of the LEQG problem. The other two players result from the duality relation. They appear due to the reduction of the problem from LEQG to LQG. We have also investigated sufficient conditions to have an LQG game problem. While in the recursive procedure the data for the model and the problem setup can be chosen freely, the state noise and the randomization should satisfy certain conditions that, at the generic time $t$, involve the previously computed coefficient $P_{t+1}$ in the leading term of the quadratic value function $V_{t+1}(x_{t+1})$. This is one of the reasons why we considered the system and randomization noises $\Lambda_t$ and $\Xi_t$ as time-dependent. The coefficient $P_t$ is deterministic and can be determined \emph{} via the backward recursion \eqref{eq:recursion:P} in Theorem \ref{theo:main:recursions} based on the estimated model data. Hence, the above-mentioned conditions can actually be checked when, given the other data, one chooses/verifies the state noise and the randomization intensity. Essentially, the state noise should not be too diffuse if $P_t$ is large. At the same time, the randomization should have no bias and not be more diffuse than a bound given by $P_t$, the risk sensitivity parameter $\theta$, and the control weight factor $N_t$.

%
%


\end{document}